\documentclass[twocolumn]{article}
\usepackage{amssymb,MnSymbol,times}
\usepackage{amsthm,amsmath}
\usepackage{cite}

\title{\huge Joint signature of two or more systems with applications to multistate systems made up of two-state components}

\author{Jean-Luc Marichal\footnote{Corresponding author: Jean-Luc Marichal is with the Mathematics Research Unit, University of Luxembourg, Maison du Nombre, 6, avenue de la Fonte, L-4364 Esch-sur-Alzette, Luxembourg. Email: jean-luc.marichal[at]uni.lu}\and Pierre Mathonet\footnote{University of Li\`ege, Department of Mathematics, All\'ee de la D\'ecouverte, 12 - B37, B-4000 Li\`ege, Belgium. Email: p.mathonet[at]ulg.ac.be} \and Jorge Navarro\footnote{Facultad de Matem\'aticas, Universidad de Murcia, 30100 Murcia, Spain. Email: jorgenav[at]um.es} \and Christian Paroissin\footnote{CNRS / Univ Pau \& Pays Adour, Laboratoire de Math\'ematiques et de leurs Applications de Pau -- F\'ed\'eration IPRA, UMR 5142, 64000 Pau, France. Email: christian.paroissin[at]univ-pau.fr}}

\date{Revised, April 15, 2017} 

\begin{document}

\theoremstyle{plain}
\newtheorem{theorem}{Theorem}[section]
\newtheorem{lemma}[theorem]{Lemma}
\newtheorem{proposition}[theorem]{Proposition}
\newtheorem{corollary}[theorem]{Corollary}
\newtheorem{fact}[theorem]{Fact}
\newtheorem*{main}{Main Theorem}

\theoremstyle{definition}
\newtheorem{definition}[theorem]{Definition}
\newtheorem{example}[theorem]{Example}
\newtheorem{remark}[theorem]{Remark}
\newtheorem*{DP}{Decomposition Principle}

\renewcommand{\S}{\mathcal{S}}

\newcommand{\N}{\mathbb{N}}                     
\newcommand{\R}{\mathbb{R}}                     
\newcommand{\bfa}{\mathbf{a}}
\newcommand{\bfx}{\mathbf{x}}
\newcommand{\bfu}{\mathbf{u}}
\newcommand{\bfv}{\mathbf{v}}
\newcommand{\bfy}{\mathbf{y}}
\newcommand{\bfz}{\mathbf{z}}
\newcommand{\bfb}{\mathbf{b}}
\newcommand{\tk}{{\mathcal T}_k}
\newcommand{\uno}{\boldsymbol 1}

\newcommand{\calS}{\mathcal{S}}
\newcommand{\Ind}{\mathrm{Ind}}

\def\ang#1{\langle{#1}\rangle}

\maketitle

\section*{Abstract}

The structure signature of a system made up of $n$ components having continuous and i.i.d.\ lifetimes was defined in the eighties by Samaniego as the $n$-tuple whose $k$-th coordinate is the probability that the $k$-th
component failure causes the system to fail. More recently, a bivariate version of this concept was considered as follows. The joint structure signature of a pair of systems built on a common set of components having continuous and i.i.d.\ lifetimes is a square matrix of order $n$ whose $(k,l)$-entry is the probability that the $k$-th failure causes the first system to fail and the $l$-th failure causes the second system to fail. This concept was successfully used to derive a signature-based decomposition of the joint reliability of the two systems. In the first part of this paper we provide an explicit formula to compute the joint structure signature of two or more systems and extend this formula to the general non-i.i.d.\ case, assuming only that the distribution of the component lifetimes has no ties. We also provide and discuss a necessary and sufficient condition on this distribution for the joint reliability of the systems to have a signature-based decomposition. In the second part of this paper we show how our results can be efficiently applied to the investigation of the reliability and signature of multistate systems made up of two-state components. The key observation is that the structure function of such a multistate system can always be additively decomposed into a sum of classical structure functions. Considering a multistate system then reduces to considering simultaneously several two-state systems.

\medskip

\noindent\textbf{Keywords:} Reliability, semicoherent system, dependent lifetimes, system signature, system joint signature, multistate system.

\medskip

\noindent\emph{2010 Mathematics Subject Classification:} 62N05, 90B25, 94C10.

\section{Introduction}

Consider a system $\S=(C,\phi,F)$, where $C$ is a set $[n]=\{1,\ldots,n\}$ of nonrepairable components, $\phi\colon\{0,1\}^n\to\{0,1\}$ is a structure function, and $F$ is the joint c.d.f.\ of the component lifetimes $T_1,\ldots,T_n$, defined by
$$
F(t_1,\ldots,t_n) ~=~ \Pr(T_1\leqslant t_1,\ldots,T_n\leqslant t_n)\, ,\quad t_1,\ldots,t_n\geqslant 0.
$$
We assume that the system $\mathcal{S}$ (or equivalently, its structure function $\phi$) is \emph{semicoherent}, which means that the function $\phi$ is nondecreasing in each variable and satisfies the conditions $\phi(\mathbf{0})=0$ and $\phi(\mathbf{1})=1$, where $\mathbf{0}=(0,\ldots,0)$ and $\mathbf{1}=(1,\ldots,1)$. As usual we say that the system $\mathcal{S}$ is \emph{coherent} if it is semicoherent and, additionally, every of its components is relevant (i.e., the function $\phi$ is nonconstant in each of its variables). We also assume throughout that the joint c.d.f.\ $F$ of the component lifetimes has no ties, which means that $\Pr(T_i=T_j)=0$ for all distinct $i,j\in C$.

Samaniego \cite{Sam85} defined the \emph{signature} of any system $\S$ whose components have continuous and i.i.d.\ lifetimes as the $n$-tuple $\mathbf{s}=(s_1,\ldots,s_n)$ whose $k$-th coordinate is the probability that
the $k$-th component failure causes the system to fail. In other words, we have
$$
s_k ~=~ \Pr(T_\S=T_{k:n}),\qquad k=1,\ldots,n,
$$
where $T_\S$ is the system lifetime and $T_{k:n}$ is the $k$-th smallest component lifetime, that is, the $k$-th order statistic of the component lifetimes.

Boland \cite{Bol01} showed that $s_k$ can be computed only from the structure function $\phi$ by means of the formula
\begin{equation}\label{eq:asad678}
s_k ~=~ \sum_{\textstyle{A\subseteq C\atop |A|=n-k+1}}\frac{1}{{n\choose |A|}}\,\phi(A)-\sum_{\textstyle{A\subseteq C\atop
|A|=n-k}}\frac{1}{{n\choose |A|}}\,\phi(A).
\end{equation}
Here and throughout we identify Boolean tuples $\bfx\in\{0,1\}^n$ and subsets $A\subseteq [n]$ by setting $x_i=1$ if and only if $i\in A$. In particular, $|A|$ denotes the cardinality of $A$ and $|\bfx|$ denotes the number $\sum_ix_i$. We also use the same symbol to denote both a function $f\colon\{0,1\}^n\to\R$ and the corresponding set function $f\colon 2^{[n]}\to\R$ interchangeably. For instance, we write $\phi(\mathbf{0})=\phi(\varnothing)$ and $\phi(\mathbf{1})=\phi([n])$.

Equation~\eqref{eq:asad678} clearly shows that the signature $\mathbf{s}$ is independent of the joint distribution $F$ of the component lifetimes. It depends only on $n$ and $\phi$ and has an important combinatorial meaning (see \cite{Bol01}). For this reason it is often called the structure signature of the system. In this paper we will therefore regard the \emph{structure signature} of the system as the $n$-tuple $\mathbf{s}=(s_1,\ldots,s_n)$ whose $k$-th coordinate is given by Eq.~\eqref{eq:asad678}. Thus defined, this concept always exists as a purely combinatorial object even in the general non-i.i.d.\ case.

The original definition of signature can be easily extended to the general non-i.i.d.\ case where it is assumed only that the joint distribution function $F$ has no ties. The \emph{probability signature} of a system $\S$ is the $n$-tuple $\mathbf{p}=(p_1,\ldots,p_n)$ whose $k$-th coordinate is defined by
$$
p_k ~=~ \Pr(T_\S=T_{k:n}),\qquad k=1,\ldots,n.
$$
By definition the identity $\mathbf{p}=\mathbf{s}$ holds (which means that the probability $\Pr(T_\S=T_{k:n})$ is exactly given by Boland's formula \eqref{eq:asad678} for $k=1,\ldots,n$) whenever the component lifetimes are i.i.d.\ and continuous. Actually, it was shown in \cite{NavRyc07} that this identity still holds whenever the component lifetimes are exchangeable (i.e., the function $F$ is symmetric) and absolutely continuous. However, it was also observed \cite{MarMat11,MarMatWal11,NavSpiBal10,Spi08} that in general the probability signature $\mathbf{p}$ depends on the joint c.d.f.\ $F$ of the component lifetimes and that this dependence is captured by means of the \emph{relative quality function} $q\colon 2^{[n]}\to [0,1]$, which is defined as
$$
q(A) ~=~ \Pr\Big(\max_{i\in C\setminus A}T_i < \min_{j\in A}T_j\Big), \qquad A\subseteq C,
$$
with the convention that $q(\varnothing)=q(C)=1$. The $k$-th coordinate of the probability signature is then given by the formula (see \cite{MarMat11})
\begin{equation}\label{eq:sdf67}
p_k ~=~ \sum_{\textstyle{A\subseteq C\atop |A|=n-k+1}}q(A)\,\phi(A)-\sum_{\textstyle{A\subseteq C\atop |A|=n-k}}q(A)\,\phi(A).
\end{equation}
which generalizes Boland's formula \eqref{eq:asad678} under the sole assumption that the joint distribution function $F$ has no ties.

\begin{remark}\label{rem:sp4}
Comparing formulas \eqref{eq:asad678} and \eqref{eq:sdf67} shows that we still have $\mathbf{p}=\mathbf{s}$ whenever $q(A)$ reduces to $1/{n\choose |A|}$ for every $A\subseteq C$. In the general dependent case when the latter condition does not hold, both $n$-tuples $\mathbf{p}$ and $\mathbf{s}$ still exist but no longer coincide.
\end{remark}

Navarro et al.~\cite{NavSamBal10,NavSamBal13} proposed to analyze the joint behavior of several systems built on a common set of components. They motivated their analysis by real-life applications based on computer networks. From a mathematical viewpoint, such situations consist in considering simultaneously $m$ semicoherent systems
$$
\S_1~=~(C,\phi_1,F){\,},~\ldots{\,},~\S_m~=~(C,\phi_m,F)
$$
having a common set $C$ of components and a common joint distribution $F$ of lifetimes.

To simplify our presentation we will henceforth restrict ourselves to the case of two systems. However, our definitions and results can be easily and naturally extended to an arbitrary number of systems.

Let us consider a basic example. The symbols $\wedge$ and $\vee$ represent the $\min$ and $\max$ functions, respectively.

\begin{example}\label{ex:1}
Consider the systems $\S_1=(C,\phi_1,F)$ and $\S_2=(C,\phi_2,F)$, where $C$ is a set of four components and the structure functions $\phi_1$ and $\phi_2$ are given by
$$
\phi_1(x_1,x_2,x_3,x_4) ~=~ x_1\wedge x_2{\,},
$$
and
$$
\phi_2(x_1,x_2,x_3,x_4) ~=~ (x_2\vee x_3)\wedge x_4{\,},
$$
as depicted in Figures \ref{figure1} and \ref{figure2}. Since $\S_1$ is a series system made up of components 1 and 2, the function $\phi_1$ is constant with respect to its third and fourth variables. Similarly, we see that $\phi_2$ is constant with respect to its first variable.
\end{example}

\setlength{\unitlength}{4.5ex}
\begin{figure}[htbp]
\begin{minipage}[t]{0.45\textwidth}
\begin{center}
\begin{picture}(7,3)
\put(2,1){\framebox(1,1){1}}\put(4,1){\framebox(1,1){2}}%
\put(1,1.5){\circle*{0.15}}\put(6,1.5){\circle*{0.15}}%
\multiput(1,1.5)(2,0){3}{\line(1,0){1}}
\end{picture}
\caption{System $\S_1$}\label{figure1}
\end{center}
\end{minipage}

\vspace{5ex}

\begin{minipage}[t]{0.45\textwidth}
\begin{center}
\begin{picture}(7,3)
\put(2,0.25){\framebox(1,1){3}}\put(2,1.75){\framebox(1,1){2}}\put(5,1){\framebox(1,1){4}}%
\put(0,1.5){\circle*{0.15}}\put(7,1.5){\circle*{0.15}}%
\multiput(1,0.75)(3,0){2}{\line(0,1){1.5}}%
\put(0,1.5){\line(1,0){1}}%
\multiput(1,0.75)(2,0){2}{\line(1,0){1}}\multiput(1,2.25)(2,0){2}{\line(1,0){1}}\multiput(4,1.5)(2,0){2}{\line(1,0){1}}%
\end{picture}
\caption{System $\S_2$}\label{figure2}
\end{center}
\end{minipage}
\end{figure}

\begin{remark}
Example~\ref{ex:1} enables us to emphasize the need to consider semicoherent systems instead of coherent systems. Indeed, by doing so we can consider several systems that share only some (and not all) components.
\end{remark}

Under the assumption that the lifetimes are i.i.d.\ and continuous, Navarro et al.~\cite{NavSamBal13} defined the \emph{joint structure signature} of the systems $\S_1$ and $\S_2$ as the square matrix $\mathbf{s}$ of order $n$ whose $(k,l)$-entry is the probability
$$
s_{k,l} ~=~ \Pr(T_{\S_1}=T_{k:n}~\mbox{and}~T_{\S_2}=T_{l:n}),\quad k,l=1,\ldots,n.
$$
Thus, $s_{k,l}$ is exactly the probability that the $k$-th failure causes the system $\S_1$ to fail and
that the $l$-th failure causes the system $\S_2$ to fail in the i.i.d.\ and continuous case.

We will see in Corollary~\ref{cor:main3} that the joint structure signature is independent of the joint distribution c.d.f.\ $F$ of the component lifetimes. It depends only on $n$ and the structure functions $\phi_1$ and $\phi_2$. It is then a purely combinatorial object and, similarly to the structure signature, we will consider that this concept still exists in the general non-i.i.d.\ case.

Just as for the concept of structure signature, the concept of joint structure signature can also be extended to the general dependent setting, assuming only that the function $F$ has no ties. Thus, we define the \emph{joint probability signature} of two systems $\S_1=(C,\phi_1,F)$ and $\S_2=(C,\phi_2,F)$ as the square matrix $\mathbf{p}$ of order $n$ whose $(k,l)$-entry is the probability
$$
p_{k,l} ~=~ \Pr(T_{\S_1}=T_{k:n}~\mbox{and}~T_{\S_2}=T_{l:n}),\quad k,l=1,\ldots,n.
$$
In general this matrix depends on both the structures of the systems and the joint c.d.f.\ $F$ of the component lifetimes.

The outline of this paper is as follows. In the first part (consisting of Sections 2 and 3) we provide a generalization of Boland's formula \eqref{eq:asad678} to joint structure signatures (Corollaries~\ref{cor:ConvForm5873s} and \ref{cor:main3}). We also introduce a bivariate version of the relative quality function (Definition~\ref{de:jrqf4}) and use it to provide an extension of Eq.~\eqref{eq:sdf67} to joint probability signatures (Theorem~\ref{thm:main3}).

In the continuous and i.i.d.\ case, the structure signature was used by Samaniego \cite{Sam85,Sam07} to derive a signature-based decomposition of the reliability
$$
\overline{F}_{\S}(t) ~=~ \Pr(T_{\S}>t),\qquad t\geqslant 0,
$$
of the system $\S$. The concept of joint structure signature was used similarly by Navarro et al.~\cite{NavSamBal13} to derive a signature-based decomposition of the joint reliability
$$
\overline{F}_{\S_1,\S_2}(t_1,t_2) ~=~ \Pr(T_{\S_1}>t_1~\mbox{and}~T_{\S_2}>t_2),\quad t_1,t_2\geqslant 0,
$$
of the systems $\S_1$ and $\S_2$. In the general non-i.i.d.\ setting, we provide and discuss a necessary and sufficient condition on the function $F$ for the joint reliability of two systems to have a signature-based decomposition (Propositions~\ref{prop:mainDEC9} and \ref{prop:mainDEC9z}).

In the second part of this paper (Sections 4--7) we fruitfully apply the results obtained in Sections 2 and 3 to the investigation of the signature and reliability of multistate systems made up of two-state components in the general dependent setting. More specifically it is shown that our results offer a general framework and an efficient tool to study such multistate systems in the dependent setting and bring simplifications and generalizations to recent results obtained, e.g., by Da and Hu~\cite{DaHu13} and Gertsbakh et al.~\cite{GerSchSpi12} on these topics.

\section{The joint probability signature}

In this section we mainly show how Eq.~\eqref{eq:asad678} can be extended to joint structure signatures and how Eq.~\eqref{eq:sdf67} can be extended to joint probability signatures.

Recall first that the \emph{tail structure signature} of a system $\S$ (a concept introduced in \cite{Bol01} and named so in \cite{GerShpSpi11}) is the $(n+1)$-tuple $\overline{\mathbf{S}}=(\overline{S}_0,\ldots,\overline{S}_n)$ defined by
\begin{equation}\label{eq:f5sffs}
\overline{S}_k ~=~ \sum_{i=k+1}^ns_i ~=~ \sum_{\textstyle{A\subseteq C\atop |A|=n-k}}\frac{1}{{n\choose |A|}}\,\phi(A){\,},\quad k=0,\ldots,n,
\end{equation}
where the latter expression immediately follows from \eqref{eq:asad678}. (Here and throughout we use the usual convention that $\sum_{i=k+1}^nx_i=0$ when $k=n$.) Conversely, the structure signature $\mathbf{s}$ can also be easily retrieved from $\overline{\mathbf{S}}$ by using \eqref{eq:asad678}, that is, by computing $s_k=\overline{S}_{k-1}-\overline{S}_k$ for $k=1,\ldots,n$. Thus, assuming that the component lifetimes are i.i.d.\ and continuous, we see that $\overline{S}_k = \Pr(T_\S>T_{k:n})$ is the probability that the system $\S$ survives beyond the $k$-th failure (with the usual convention that $T_{0:n}=0$).

Similarly, the \emph{tail probability signature} of a system $\S$ (see \cite{MarMatSpi15}) is the $(n+1)$-tuple $\overline{\mathbf{P}}=(\overline{P}_0,\ldots,\overline{P}_n)$, where
\begin{equation}\label{eq:f5sff}
\overline{P}_k ~=~ \sum_{i=k+1}^np_i ~=~ \sum_{\textstyle{A\subseteq C\atop |A|=n-k}}q(A)\,\phi(A){\,},\quad k=0,\ldots,n.
\end{equation}
Conversely, the probability signature $\mathbf{p}$ can be easily retrieved from $\overline{\mathbf{P}}$ by using \eqref{eq:sdf67}, that is, by computing $p_k=\overline{P}_{k-1}-\overline{P}_k$ for $k=1,\ldots,n$. Thus, assuming only that the function $F$ has no ties, we see that $\overline{P}_k=\Pr(T_\S>T_{k:n})$ is the probability that the system $\S$ survives beyond the $k$-th failure.

Since the tail signatures proved to be much easier to handle than the standard signatures in many computation problems, it is natural to extend these concepts to the bivariate case.

Given two systems $\S_1=(C,\phi_1,F)$ and $\S_2=(C,\phi_2,F)$ for which the component lifetimes are i.i.d.\ and continuous, the \emph{joint tail structure signature} is defined as the square matrix $\overline{\mathbf{S}}$ of order $n+1$ whose $(k,l)$-entry is the probability
$$
\overline{S}_{k,l} ~=~ \Pr(T_{\S_1}>T_{k:n}~\mbox{and}~T_{\S_2}>T_{l:n}),\quad k,l=0,\ldots,n.
$$
Similarly, assuming only that the function $F$ has no ties, the \emph{joint tail probability signature} is the square matrix $\overline{\mathbf{P}}$ of order $n+1$ whose $(k,l)$-entry is the probability
$$
\overline{P}_{k,l} ~=~ \Pr(T_{\S_1}>T_{k:n}~\mbox{and}~T_{\S_2}>T_{l:n}),\quad k,l=0,\ldots,n.
$$
Thus, $\overline{P}_{k,l}$ is the probability that the system $\S_1$ survives beyond the $k$-th failure and the system $\S_2$ survives beyond the $l$-th failure. In particular,
\begin{eqnarray*}
\overline{P}_{k,l} &=& 0, \quad\mbox{if $~k=n~$ or $~l=n$}{\,},\\
\overline{P}_{k,0} &=& \Pr(T_{\S_1}>T_{k:n}){\,},\quad k=0,\ldots,n,\\
\overline{P}_{0,l} &=& \Pr(T_{\S_2}>T_{l:n}){\,},\quad l=0,\ldots,n,\\
\overline{P}_{0,0} &=& 1.
\end{eqnarray*}

The following straightforward proposition provides the conversion formulas between the matrices $\mathbf{p}$ and $\overline{\mathbf{P}}$. The corresponding formulas between the matrices $\mathbf{s}$ and $\overline{\mathbf{S}}$ are given in Corollary~\ref{cor:ConvForm5873s}.

\begin{proposition}\label{prop:ConvForm5873}
We have
\begin{equation}\label{eq:CoFo1}
\overline{P}_{k,l} ~=~ \sum_{i=k+1}^n\sum_{j=l+1}^np_{i,j}{\,},\quad k,l=0,\ldots,n,
\end{equation}
and
\begin{equation}\label{eq:CoFo2}
p_{k,l} ~=~ \overline{P}_{k-1,l-1}-\overline{P}_{k,l-1}-\overline{P}_{k-1,l}+\overline{P}_{k,l}{\,},\quad  k,l=1,\ldots,n.
\end{equation}
\end{proposition}

%

\begin{corollary}\label{cor:ConvForm5873s}
We have
\[
\overline{S}_{k,l} ~=~ \sum_{i=k+1}^n\sum_{j=l+1}^ns_{i,j}{\,},\quad k,l=0,\ldots,n,
\]
and
\[
s_{k,l} ~=~ \overline{S}_{k-1,l-1}-\overline{S}_{k,l-1}-\overline{S}_{k-1,l}+\overline{S}_{k,l}{\,},\quad  k,l=1,\ldots,n.
\]
\end{corollary}

The conversion formulas given in Proposition~\ref{prop:ConvForm5873} and Corollary~\ref{cor:ConvForm5873s} show that all the information contained in the matrix $\mathbf{p}$ (resp.\ the matrix $\mathbf{s}$) is completely encoded in the matrix $\overline{\mathbf{P}}$ (resp.\ the matrix $\overline{\mathbf{S}}$) and vice versa.

We will now show how to compute the matrix $\overline{\mathbf{P}}$ from $\phi_1$, $\phi_2$, and $F$ (see Theorem~\ref{thm:main3}) and similarly for the matrix $\overline{\mathbf{S}}$ (see Corollary~\ref{cor:main3}). These results actually constitute direct generalizations of Eqs.~\eqref{eq:f5sffs} and \eqref{eq:f5sff}, which are the `tail' versions of Eqs.~\eqref{eq:asad678} and \eqref{eq:sdf67}, to joint structure signatures and joint probability signatures, respectively.

Let us first introduce the bivariate version of the relative quality function.

\begin{definition}\label{de:jrqf4}
The \emph{joint relative quality function} associated with the joint c.d.f.\ $F$ is the symmetric function $q\colon 2^{[n]}\times 2^{[n]}\to [0,1]$ defined by
$$
q(A,B) ~=~ \Pr\Big(\max_{i\in C\setminus A}T_i < \min_{j\in A}T_j~\mbox{ and }~\max_{i\in C\setminus B}T_i < \min_{j\in B}T_j\Big),
$$
with the convention that $q(A,\varnothing)=q(A,C)=q(A)$ for every $A\subseteq C$ and $q(\varnothing,B)=q(C,B)=q(B)$ for every $B\subseteq C$.
\end{definition}

By definition, the joint relative quality function satisfies the following properties: for every $A,B\subseteq C$, we have $q(A,B)=q(B,A)$ and $q(A,A)=q(A)$. Moreover, the number $q(A,B)$ is the probability that the best $|A|$ components are precisely those in $A$ and that the best $|B|$ components are precisely those in $B$. In particular, we have $q(A,B)=0$ whenever $A\nsubseteq B$ and $B\nsubseteq A$.

A important link between the relative quality function and the joint relative quality function is given in the following proposition.

\begin{proposition}\label{prop:5re86}
For every $A\subseteq C$ and every $l\in\{0,\ldots,n\}$ we have
\[
\sum_{\textstyle{B\subseteq C\atop |B|=l}}q(A,B) ~=~ q(A).
\]
In particular, for every $k,l\in\{0,\ldots,n\}$ we have
\[
\sum_{\textstyle{A\subseteq C\atop |A|=k}}\sum_{\textstyle{B\subseteq C\atop |B|=l}}q(A,B) ~=~ 1.
\]
\end{proposition}

The joint relative quality function can also be computed in terms of probabilities of events labeled by the permutations of $[n]=\{1,\ldots,n\}$. Indeed, denote by $\Sigma_n$ the group of permutations of $[n]$ and define the events
\[
E_\sigma ~=~ (T_{\sigma(1)}<\cdots <T_{\sigma(n)}){\,},\qquad \sigma\in \Sigma_n.
\]
Since the function $F$ has no ties, the collection of events $\{E_\sigma:\sigma\in \Sigma_n\}$ forms a partition almost everywhere of the sample space $\left[0,+\infty\right[^n$. Moreover, it is clear that
\begin{multline*}
\Big(\max_{i\in C\setminus A}T_i < \min_{j\in A}T_j~\mbox{ and }~\max_{i\in C\setminus B}T_i < \min_{j\in B}T_j\Big)\\
~=~ \bigcup_{\textstyle{\sigma\in \Sigma_n\, :\, \{\sigma(n-|A|+1),\ldots,\sigma(n)\}=A \atop
\{\sigma(n-|B|+1),\ldots,\sigma(n)\}=B}}E_{\sigma}
\end{multline*}
almost everywhere. We thus have
\begin{multline}
q(A,B)\\
 ~=~ \sum_{\textstyle{\sigma\in \Sigma_n\, :\, \{\sigma(n-|A|+1),\ldots,\sigma(n)\}=A \atop
\{\sigma(n-|B|+1),\ldots,\sigma(n)\}=B}}\Pr(T_{\sigma(1)}<\cdots <T_{\sigma(n)}).\label{eq:sadf65sdf}
\end{multline}
We now show that this expression can be easily computed in the i.i.d.\ and continuous case, or more generally whenever the events $E_\sigma$ $(\sigma \in \Sigma_n)$ are equally likely, for which we have $\Pr(E_\sigma)=1/n!$ since the function $F$ has no ties.

\begin{definition}\label{de:q0}
Define the symmetric function $q_0\colon 2^{[n]}\times 2^{[n]}\to [0,1]$ as
\begin{equation}\label{eq:qExch5}
q_0(A,B) ~=~
\begin{cases}
\frac{(n-|A|)!{\,}(|A|-|B|)!{\,}|B|!}{n!_{\mathstrut}} & \mbox{if $B\subseteq A$},\\
\frac{(n-|B|)!{\,}(|B|-|A|)!{\,}|A|!}{n!_{\mathstrut}} & \mbox{if $A\subseteq B$},\\
0 & \mbox{otherwise}.
\end{cases}
\end{equation}
\end{definition}

\begin{proposition}\label{prop:7ert}
If the events $E_\sigma$ $(\sigma \in \Sigma_n)$ are equally likely, then $q=q_0$.
\end{proposition}

\begin{proof}
We simply use Eq.~\eqref{eq:sadf65sdf} and count the relevant permutations.
\end{proof}

We now show how to compute the matrix $\overline{\mathbf{P}}$ from $\phi_1$, $\phi_2$, and $q$. This result together with Proposition~\ref{prop:ConvForm5873} clearly generalize Eq.~\eqref{eq:f5sff} to the case of two systems.

\begin{theorem}\label{thm:main3}
For every $k,l\in\{0,\ldots,n\}$ we have
\begin{equation}\label{eq:f5sffx}
\overline{P}_{k,l} ~=~ \sum_{|A|=n-k}{\,}\sum_{|B|=n-l}q(A,B)\phi_1(A)\phi_2(B).
\end{equation}
\end{theorem}

\begin{proof}
We just consider the event
\[
E_{k,l} ~=~ (T_{\S_1}>T_{k:n}~\mbox{and}~T_{\S_2}>T_{l:n})
\]
that defines $\overline{P}_{k,l}$. The fact that a given observed tuple of lifetimes $(t_1,\ldots,t_n)$ belongs to
$E_{k,l}$ depends only on the ordering of lifetimes. Therefore, for each permutation $\sigma\in \Sigma_n$, we have either
$E_\sigma\subseteq E_{k,l}$ or $E_\sigma\cap E_{k,l}=\varnothing$. Moreover, the former case occurs if and only if
\[
\phi_1(\{\sigma(k+1),\ldots,\sigma(n)\}) ~=~ \phi_2(\{\sigma(l+1),\ldots,\sigma(n)\}) ~=~ 1.
\]
We thus have
$$
\overline{P }_{k,l} ~=~ \Pr(E_{k,l}) ~=~ \Pr\Bigg(\bigcup_{\textstyle{\sigma\in \Sigma_n\, :\, \phi_1(\{\sigma(k+1),\ldots,\sigma(n)\})=1 \atop \phi_2(\{\sigma(l+1),\ldots,\sigma(n)\})=1}}E_\sigma\Bigg),
$$
that is,
\begin{eqnarray*}
\overline{P }_{k,l} &=& \sum_{\sigma\in \Sigma_n}\Pr(E_\sigma)~\phi_1(\{\sigma(k+1),\ldots,\sigma(n)\})\\
&& \times~
\phi_2(\{\sigma(l+1),\ldots,\sigma(n)\}).
\end{eqnarray*}
To compute this sum we group the terms for which $\{\sigma(k+1),\ldots,\sigma(n)\}$ is a given set $A$ of cardinality $n-k$ and $\{\sigma(l+1),\ldots,\sigma(n)\}$ is a given set $B$ of cardinality $n-l$ (with a necessary inclusion relation) and then sum over all the possibilities for such $A$'s and $B$'s. The result then follows from Eq.~\eqref{eq:sadf65sdf}.
\end{proof}

As mentioned above, Eq.~\eqref{eq:f5sffx} generalizes Eq.~\eqref{eq:f5sff}. Moreover, if $\overline{\mathbf{P}}^{1}$ (resp.\ $\overline{\mathbf{P}}^{2}$) denotes the tail probability signature of the system $\S_1$ (resp.\ $\S_2$), from Eqs.~\eqref{eq:f5sff} and \eqref{eq:f5sffx} it follows that
\begin{eqnarray*}
\overline{P}_{k,0} ~=~ \overline{P}_k^{1} &&\text{for $k=0,\ldots,n$},\label{eq:s675f1}\\
\overline{P}_{0,l} ~=~ \overline{P}_l^{2} &&\text{for $l=0,\ldots,n$}.\label{eq:s675f2}
\end{eqnarray*}
Also, if $\mathbf{p}^1$ (resp.\ $\mathbf{p}^2$) denotes the probability signature of the system $\S_1$ (resp.\ $\S_2$), we clearly have
\begin{eqnarray}
\sum_{l=1}^n p_{k,l} ~=~ p_k^1 &&\text{for $k=1,\ldots,n$},\label{eq:s675f1}\\
\sum_{k=1}^n p_{k,l} ~=~ p_l^2 &&\text{for $l=1,\ldots,n$}.\label{eq:s675f2}
\end{eqnarray}

From Theorem~\ref{thm:main3} we immediately derive the following corollary, which concerns the joint tail structure signature.

\begin{corollary}\label{cor:main3}
For every $k,l\in\{0,\ldots,n\}$ we have
\begin{equation}\label{eq:main3cor}
\overline{S}_{k,l} ~=~ \sum_{|A|=n-k}{\,}\sum_{|B|=n-l}q_0(A,B)\phi_1(A)\phi_2(B),
\end{equation}
\end{corollary}

\begin{remark}\label{rem:as5f}
Corollary~\ref{cor:ConvForm5873s} together with Corollary~\ref{cor:main3} generalize Eq.~\eqref{eq:f5sffs} (i.e., the `tail' version of Boland's formula \eqref{eq:asad678}) to the case of two systems. This shows that even though the matrices $\mathbf{s}$ and $\overline{\mathbf{S}}$ are defined in the i.i.d.\ case, they are also combinatorial objects associated to the systems. Throughout this paper we will thus consider these matrices as defined in Corollaries~\ref{cor:ConvForm5873s} and \ref{cor:main3}, in full accordance with the univariate case.
\end{remark}

Let us now consider a basic example to help the reader get familiar with these results.

\begin{example}
Applying Corollary~\ref{cor:main3} and then Corollary~\ref{cor:ConvForm5873s} to Example~\ref{ex:1}, we obtain the following matrices
$$
\overline{\mathbf{S}} ~=~ \frac{1}{12}
\begin{bmatrix}
12 & 9 & 4 & 0 & 0\\
6 & 3 & 1 & 0 & 0\\
2 & 1 & 0 & 0 & 0\\
0 & 0 & 0 & 0 & 0\\
0 & 0 & 0 & 0 & 0
\end{bmatrix}
\quad\mbox{and}\quad
\mathbf{s} ~=~ \frac{1}{12}
\begin{bmatrix}
0 & 3 & 3 & 0\\
2 & 1 & 1 & 0\\
1 & 1 & 0 & 0\\
0 & 0 & 0 & 0
\end{bmatrix}
.
$$
\end{example}

\begin{remark}
Formula~\eqref{eq:s675f1} shows that, by adding the entries in each row of the matrix $\mathbf{s}$, we obtain the structure signature $\mathbf{s}^1$ (with four coordinates) of the system $\S_1$. Applying this fact to the example above, we immediately obtain $\mathbf{s}^1=\frac{1}{12}(6,4,2,0)$. However, it is important to remember that $\S_1$ has four components, two of which are null or irrelevant. As detailed in \cite{NavSamBalBha08}, from this signature it is not difficult to compute the signature of the (coherent) system obtained from $\S_1$ by ignoring the irrelevant components.
\end{remark}

It is clear that Theorem~\ref{thm:main3} easily generalizes to the case of $m$ systems
$$
\S_1=(C,\phi_1,F),\ldots,\S_m=(C,\phi_m,F).
$$
Indeed, using an obvious notation, for $k_1,\ldots,k_m\in\{0,\ldots,n\}$ we then have
\begin{multline*}
\overline{P}_{k_1,\ldots,k_m}\\
~=~ \sum_{|A_1|=n-k_1}\cdots\sum_{|A_m|=n-k_m}q(A_1,\ldots,A_m){\,}\phi_1(A_1)\cdots\phi_m(A_m).
\end{multline*}
The same applies to Corollary~\ref{cor:main3}. Moreover, if there exists a permutation $\sigma\in \Sigma_m$ such that $A_{\sigma(m)}\subseteq\cdots\subseteq A_{\sigma(1)}$, then
\begin{multline*}
q_0(A_1,\ldots,A_m)\\
~=~ \frac{(n-|A_{\sigma(1)}|)!{\,}(|A_{\sigma(1)}|-|A_{\sigma(2)}|)!{\,}\cdots{\,} |A_{\sigma(m)}|!}{n!}{\,}.
\end{multline*}
Otherwise, $q_0(A_1,\ldots,A_m)=0$.

\section{Signature-based decomposition of the joint reliability}

Navarro et al.~\cite{NavSamBal13} recently provided a signature-based decomposition of the joint reliability function of two systems $\S_1=(C,\phi_1,F)$ and $\S_2=(C,\phi_2,F)$ whose components have continuous and i.i.d.\ lifetimes.

In this section we provide the generalization of this decomposition to arbitrary distributions of the component lifetimes. More precisely, we give necessary and sufficient conditions on the joint c.d.f.\ $F$ for the signature-based decomposition to hold for every pair of systems. Note that in this general setting we consider both the decomposition with respect to the structure signature (seen as a combinatorial object; see Corollaries~\ref{cor:ConvForm5873s} and \ref{cor:main3} and Remark~\ref{rem:as5f}) and the decomposition with respect to the probability signature.


For every $j\in C$ and every $t\geqslant 0$, let us denote by $X_j(t)=\mathrm{Ind}(T_j>t)$ the state variable of component $j$ at time $t$. Let us also consider the state vector $\mathbf{X}(t)=(X_1(t),\ldots,X_n(t))$ at time $t\geqslant 0$.

It was proved in \cite{MarMatWal11} (see also \cite{DukMar08,NavSamBalBha08,Sam85} for earlier works) that the reliability function of any system $\S=(C,\phi,F)$ is given by
\begin{equation}\label{eq:sf57}
\overline{F}_{\mathcal{S}}(t) ~=~ \sum_{\bfx\in\{0,1\}^n}\phi(\bfx){\,}\Pr(\mathbf{X}(t)=\bfx),\qquad t\geqslant 0,
\end{equation}
and that if the state variables $X_1(t),\ldots,X_n(t)$ are exchangeable for a given $t\geqslant 0$, which means that
$$
\Pr(\mathbf{X}(t)=\mathbf{x}) ~=~ \Pr(\mathbf{X}(t)=\sigma(\mathbf{x})){\,},\quad \bfx\in\{0,1\}^n,~\sigma\in \Sigma_n,
$$
where $\sigma(\mathbf{x})=(x_{\sigma(1)},\ldots,x_{\sigma(n)})$, then the reliability $\overline{F}_\S(t)$ of the system at time $t\geqslant 0$ admits the following signature-based decomposition:
\begin{equation}\label{eq:837gf465}
\overline{F}_{\mathcal{S}}(t) ~=~ \sum_{k=1}^n s_k{\,}\overline{F}_{k:n}(t){\,},
\end{equation}
where $s_k$ is given by Eq.~\eqref{eq:asad678} and $\overline{F}_{k:n}(t)=\Pr(T_{k:n}>t)$.

We actually have the following stronger result, which essentially states that the exchangeability of the state variables is a necessary and sufficient condition for the signature-based decomposition of the reliability to hold.

\begin{proposition}[see {\cite[Thm 4]{MarMatWal11}}]\label{prop:T1}
Let $t\geqslant 0$ be fixed. If the state variables $X_1(t),\ldots,X_n(t)$ are exchangeable, then \eqref{eq:837gf465} holds for any semicoherent system $\mathcal{S}$. Conversely, if $n\geqslant 3$ and if \eqref{eq:837gf465} holds for any coherent system $\mathcal{S}$, then the state variables $X_1(t),\ldots,X_n(t)$ are exchangeable.
\end{proposition}

The following proposition gives a variant of Proposition~\ref{prop:T1} in which the reliability function is decomposed with respect to the probability signature.

\begin{proposition}[see {\cite[Thm 6]{MarMatWal11}}]\label{prop:T2}
Let $t\geqslant 0$ be fixed. If the joint c.d.f.\ $F$ satisfies the condition
\begin{equation}\label{eq:35gh}
\Pr(\mathbf{X}(t)=\mathbf{x}) ~=~ q(\mathbf{x}){\,}\sum_{\textstyle{\mathbf{u}\in\{0,1\}^n\atop |\mathbf{u}|=|\mathbf{x}|}}\Pr(\mathbf{X}(t)=\mathbf{u})
\end{equation}
for any $\bfx\in\{0,1\}^n$, then we have
\begin{equation}\label{eq:837gf4652}
\overline{F}_{\mathcal{S}}(t) ~=~ \sum_{k=1}^n p_k{\,}\overline{F}_{k:n}(t)
\end{equation}
for any semicoherent system $\mathcal{S}$. Conversely, if $n\geqslant 3$ and if \eqref{eq:837gf4652} holds for any coherent system $\mathcal{S}$, then the joint c.d.f.\ $F$ satisfies condition \eqref{eq:35gh} for any $\bfx\in\{0,1\}^n$.
\end{proposition}

We now generalize Eq.~\eqref{eq:sf57} and Propositions~\ref{prop:T1} and \ref{prop:T2} to the case where two systems $\S_1=(C,\phi_1,F)$ and $\S_2=(C,\phi_2,F)$ are simultaneously considered. To this aim we will consider the joint distribution of the state vectors $\mathbf{X}(t_1)$ and $\mathbf{X}(t_2)$ at times $t_1\geqslant 0$ and $t_2\geqslant 0$. To simplify the notation we regard these two vectors together as a single object, namely a $2\times n$ random array ${\mathbf{X}(t_1)\choose \mathbf{X}(t_2)}$.

Let us first express the joint reliability function in terms of the component state vectors. This immediate result merely generalizes Eq.~\eqref{eq:sf57} to the case of two systems.

\begin{proposition}\label{prop:f5fd}
We have
\begin{equation}\label{eq:f5fd00}
\overline{F}_{\mathcal{S}_1,\mathcal{S}_2}(t_1,t_2)~=~\sum_{\bfx,\bfy\in\{0,1\}^n}\phi_1(\bfx)\phi_2(\bfy){\,}
\Pr({\mathbf{X}(t_1)\choose \mathbf{X}(t_2)}={\mathbf{x}\choose\mathbf{y}}).
\end{equation}
\end{proposition}

\begin{proof}
We have
\begin{eqnarray*}
  \overline{F}_{\mathcal{S}_1,\mathcal{S}_2}(t_1,t_2) &=& \Pr(\phi_1(\mathbf{X}(t_1))=1 ~\mbox{and}~\phi_2(\mathbf{X}(t_2))=1) \\
  &=& \sum_{\textstyle{\bfx\in\{0,1\}^n\atop \phi_1(\bfx)=1}}{\,}\sum_{\textstyle{\bfy\in\{0,1\}^n\atop \phi_2(\bfy)=1}}\Pr({\mathbf{X}(t_1)\choose \mathbf{X}(t_2)}={\mathbf{x}\choose\mathbf{y}}),
\end{eqnarray*}
which proves the result.
\end{proof}

By applying Proposition~\ref{prop:f5fd} to the joint reliability function of a $k$-out-of-$n$ system and an $l$-out-of-$n$ system, namely
$$
\overline{F}_{k:n,l:n}(t_1,t_2) ~=~ \Pr(T_{k:n}>t_1 ~\mbox{and}~T_{l:n}>t_2)
$$
($k,l=1,\ldots,n$) we immediately obtain the following corollary.

\begin{corollary}\label{cor:sf6}
We have
$$
\overline{F}_{k:n,l:n}(t_1,t_2) ~=~ \sum_{\textstyle{\bfx\in\{0,1\}^n\atop |\bfx|\geqslant n-k+1}}{\,}\sum_{\textstyle{\bfy\in\{0,1\}^n\atop |\bfy|\geqslant n-l+1}}\Pr({\mathbf{X}(t_1)\choose \mathbf{X}(t_2)}={\mathbf{x}\choose\mathbf{y}}).
$$
\end{corollary}

Let us now generalize Proposition~\ref{prop:T1} to the joint reliability function of two systems $\S_1=(C,\phi_1,F)$ and $\S_2=(C,\phi_2,F)$. To this extent we naturally consider the joint distribution of the component states at times $t_1\geqslant 0$ and $t_2\geqslant 0$, that is, the joint distribution of the $2\times n$ random array ${\mathbf{X}(t_1)\choose \mathbf{X}(t_2)}$. The columns of this array are the $\{0,1\}^2$-valued random variables
$$
V_1(t_1,t_2) ~=~ {X_1(t_1)\choose X_1(t_2)}~,~\ldots ~,~ V_n(t_1,t_2) ~=~ {X_n(t_1)\choose X_n(t_2)}.
$$
They are \emph{exchangeable} if
$$
(V_{\sigma(1)}(t_1,t_2),\ldots,V_{\sigma(n)}(t_1,t_2))
$$
and
$$
(V_1(t_1,t_2),\ldots,V_n(t_1,t_2))
$$
have the same distribution for any permutation $\sigma\in \Sigma_n$.
Equivalently, this exchangeability property holds if and only if
\begin{equation}\label{eq:cond.12}
\Pr({\mathbf{X}(t_1)\choose \mathbf{X}(t_2)}={\mathbf{x}\choose\mathbf{y}}) ~=~ \Pr({\mathbf{X}(t_1)\choose \mathbf{X}(t_2)}={\sigma(\mathbf{x})\choose\sigma(\mathbf{y})}){\,}
\end{equation}
for any $\bfx,\bfy\in\{0,1\}^n$ and any $\sigma\in \Sigma_n$.

\begin{lemma}\label{lemma:et5h}
Let $t_1,t_2\geqslant 0$ be fixed. Condition \eqref{eq:cond.12} holds for any $\bfx,\bfy\in\{0,1\}^n$ and any $\sigma\in \Sigma_n$ if and only if for any $\bfx,\bfy\in\{0,1\}^n$ we have
\begin{multline}
\Pr({\mathbf{X}(t_1)\choose \mathbf{X}(t_2)}={\mathbf{x}\choose\mathbf{y}})\\
~=~ q_0(\bfx,\bfy)\sum_{\textstyle{\mathbf{u}\in\{0,1\}^n\atop |\mathbf{u}|=|\bfx|}}\sum_{\textstyle{\mathbf{v}\in\{0,1\}^n\atop |\mathbf{v}|=|\bfy|}}\Pr({\mathbf{X}(t_1)\choose \mathbf{X}(t_2)}={\mathbf{u}\choose\mathbf{v}}).\label{eq:as9f7}
\end{multline}
\end{lemma}

\begin{proof}
Since the components are not repairable and by definition of $q_0(\bfx,\bfy)$, each side of Eqs.~\eqref{eq:cond.12} and~\eqref{eq:as9f7} reduces to zero whenever $\bfx\nleqslant\bfy$ and $\bfy\nleqslant\bfx$ (i.e., $\bfx$ and $\bfy$ are not ordered by inclusion). We then can assume for instance that $\bfx\geqslant\bfy$. The other case can be dealt with similarly.

Condition~\eqref{eq:cond.12} can be equivalently rewritten as
$$
\Pr({\mathbf{X}(t_1)\choose \mathbf{X}(t_2)}={\mathbf{x}\choose\mathbf{y}}) ~=~ \frac{1}{n!}\sum_{\sigma\in \Sigma_n}~\Pr({\mathbf{X}(t_1)\choose \mathbf{X}(t_2)}={\sigma(\mathbf{x})\choose\sigma(\mathbf{y})}){\,}.
$$

The right-hand side of this equation can also be rewritten as
$$
\frac{1}{n!}~\sum_{\bfu,\bfv\in\{0,1\}^n}n(\bfx,\bfy,\bfu,\bfv){\,}\Pr({\mathbf{X}(t_1)\choose \mathbf{X}(t_2)}={\bfu\choose\bfv}),
$$
where $n(\bfx,\bfy,\bfu,\bfv)$ represents the number of permutations $\sigma\in \Sigma_n$ such that ${\sigma(\mathbf{x})\choose\sigma(\mathbf{y})}={\bfu\choose\bfv}$. Clearly, we have $n(\bfx,\bfy,\bfu,\bfv)=0$ if $|\bfu|\neq |\bfx|$, or $|\bfv|\neq |\bfy|$, or $\bfu\ngeqslant\bfv$. Otherwise, the matrices ${\bfx\choose\bfy}$ and ${\bfu\choose\bfv}$ have the same number of columns of each possible type and hence we obtain
$$
n(\bfx,\bfy,\bfu,\bfv) ~=~ (n-|\bfx|)!{\,}(|\bfx|-|\bfy|)!{\,}|\bfy|!{\,}.
$$
This completes the proof.
\end{proof}

\begin{remark}\label{rem: abc01}
It follows directly from the proof of the previous lemma that condition \eqref{eq:cond.12} holds for any nonzero
$\bfx,\bfy\in\{0,1\}^n$ and any $\sigma\in \Sigma_n$ if and only if \eqref{eq:as9f7} holds for any nonzero
$\bfx,\bfy\in\{0,1\}^n$.
\end{remark}

We can now state the generalization of Proposition~\ref{prop:T1} to the case of two systems.

\begin{proposition}\label{prop:mainDEC9}
Let $t_1,t_2\geqslant 0$ be fixed. If the joint c.d.f.\ $F$ satisfies condition \eqref{eq:cond.12} for any nonzero $\bfx,\bfy\in\{0,1\}^n$, then we have
\begin{equation}\label{eq:sfd76}
\overline{F}_{\mathcal{S}_1,\mathcal{S}_2}(t_1,t_2) ~=~ \sum_{k=1}^n{\,}\sum_{l=1}^n s_{k,l}{\,}\overline{F}_{k:n,l:n}(t_1,t_2)
\end{equation}
for any semicoherent systems $\mathcal{S}_1$ and $\mathcal{S}_2$, where $s_{k,l}$ is defined in
Corollaries~\ref{cor:ConvForm5873s} and \ref{cor:main3}. Conversely, if $n\geqslant 3$ and if \eqref{eq:sfd76} holds for any coherent systems $\mathcal{S}_1$ and $\mathcal{S}_2$ (at times $t_1,t_2$), then
the joint c.d.f.\ $F$ satisfies condition \eqref{eq:cond.12} for any nonzero $\bfx,\bfy\in\{0,1\}^n$.
\end{proposition}

\begin{proof}
The left-hand side of \eqref{eq:sfd76} is given by \eqref{eq:f5fd00}.
Let us compute the right-hand side. For any sequence $x_k$, define the sequence $(\Delta x)_k=x_{k+1}-x_k$.
For two sequences $x_k$ and $y_k$ such that $x_0y_0=x_ny_n=0$, we then clearly have
\begin{equation}\label{eq:SumPP}
\sum_{k=1}^n(\Delta x)_{k-1}y_k = -\sum_{k=0}^{n-1} x_k(\Delta y)_k.
\end{equation}

For any double sequence $x_{k,l}$ $(k,l=0,\ldots,n)$, define $(\Delta_1x)_{k,l}=x_{k+1,l}-x_{k,l}$ and $(\Delta_2x)_{k,l}=x_{k,l+1}-x_{k,l}$. By Corollary~\ref{cor:ConvForm5873s} we then clearly have $s_{k,l}=(\Delta_1\Delta_2\overline{S})_{k-1,l-1}$ for $k,l=1,\ldots,n$.

On the other hand, for the sequence $f_{k,l}=\overline{F}_{k:n,l:n}(t_1,t_2)$, by Corollary~\ref{cor:sf6} we observe that
\begin{equation}\label{eq:wr7rw}
(\Delta_1\Delta_2f)_{k,l}=\sum_{\textstyle{\mathbf{u}\in\{0,1\}^n\atop |\mathbf{u}|=n-k}}\sum_{\textstyle{\mathbf{v}\in\{0,1\}^n\atop |\mathbf{v}|=n-l}}\Pr({\mathbf{X}(t_1)\choose \mathbf{X}(t_2)}={\bfu\choose\bfv}).
\end{equation}
Now, observing that $\overline{S}_{k,l}=0$ whenever $k=n$ or $l=n$ and defining $\overline{F}_{k:n,l:n}(t_1,t_2)=0$ whenever $k=0$ or $l=0$, by two applications of (\ref{eq:SumPP}) we can rewrite the right-hand side of (\ref{eq:sfd76}) as
$$
\sum_{k=1}^n\sum_{l=1}^n (\Delta_1\Delta_2\overline{S})_{k-1,l-1} f_{k,l} ~=~ \sum_{k=0}^{n-1}\sum_{l=0}^{n-1} \overline{S}_{k,l}(\Delta_1\Delta_2f)_{k,l}.
$$
Using Eqs.~\eqref{eq:main3cor} and \eqref{eq:wr7rw}, this double sum immediately becomes
\begin{multline*}
\sum_{k=0}^{n-1}\sum_{l=0}^{n-1} \bigg(\sum_{\textstyle{\bfx\in\{0,1\}^n\atop |\bfx|=n-k}}\sum_{\textstyle{\bfy\in\{0,1\}^n\atop |\bfy|=n-l}}q_0(\bfx,\bfy)\phi_1(\bfx)\phi_2(\bfy)\bigg)\\
\times \bigg(\sum_{\textstyle{\mathbf{u}\in\{0,1\}^n\atop |\mathbf{u}|=n-k}}\sum_{\textstyle{\mathbf{v}\in\{0,1\}^n\atop |\mathbf{v}|=n-l}}\Pr({\mathbf{X}(t_1)\choose \mathbf{X}(t_2)}={\bfu\choose\bfv})\bigg)
\end{multline*}
or equivalently (observing that $\phi_1(\bfx)\phi_2(\bfy)=0$ whenever $|\bfx|=|\bfy|=0$),
\begin{multline} \sum_{\bfx\in\{0,1\}^n}\sum_{\bfy\in\{0,1\}^n}q_0(\bfx,\bfy)\phi_1(\bfx)\phi_2(\bfy)\\
\times \sum_{\textstyle{\mathbf{u}\in\{0,1\}^n\atop |\mathbf{u}|=|\bfx|}}\sum_{\textstyle{\mathbf{v}\in\{0,1\}^n\atop |\mathbf{v}|=|\bfy|}}\Pr({\mathbf{X}(t_1)\choose \mathbf{X}(t_2)}={\bfu\choose\bfv}).\label{eq:11122}
\end{multline}
If $F$ satisfies condition \eqref{eq:cond.12} for any nonzero $\bfx,\bfy\in\{0,1\}^n$, then it follows from Lemma~\ref{lemma:et5h},
Remark \ref{rem: abc01}, and Proposition \ref{prop:f5fd} that \eqref{eq:sfd76} holds.

Let us now prove the converse part. Denote by $\Psi_n$ the vector space of functions $\phi\colon\{0,1\}^n\to\R$ satisfying $\phi(\mathbf{0})=0$. Fix $t_1,t_2\geqslant 0$ and define two bilinear operators $B_1$ and $B_2$ on $\Psi_n$ by letting $B_1(\phi_1,\phi_2)$ be equal to the right-hand side of \eqref{eq:f5fd00} and $B_2(\phi_1,\phi_2)$ be given by \eqref{eq:11122} for any $\phi_1,\phi_2\in\Psi_n$.
Then, if \eqref{eq:sfd76} holds for any coherent systems $\mathcal{S}_1$ and $\mathcal{S}_2$, the restrictions of $B_1$ and $B_2$ to the set of coherent structure functions are equal. Since $B_1$ and $B_2$ are bilinear operators and since, for $n\geqslant 3$, the coherent structure functions span the vector space $\Psi_n$ (see \cite[Appendix]{MarMatWal11}), the operators $B_1$ and $B_2$ are equal. In particular, for any nonzero $\bfx,\bfy\in\{0,1\}^n$ we have $B_1(\delta_{\bfx},\delta_\bfy)=B_2(\delta_{\bfx},\delta_\bfy)$, where for every $\mathbf{t}\neq\mathbf{0}$
the function $\delta_{\mathbf{t}}$ is defined by $\delta_{\mathbf{t}}(\bfz)=1$, if $\bfz=\mathbf{t}$, and $\delta_{\mathbf{t}}(\bfz)=0$, otherwise. We then obtain \eqref{eq:as9f7} for any nonzero $\bfx,\bfy\in\{0,1\}^n$.
\end{proof}

\begin{remark}
\begin{enumerate}
\item[(a)] We would like to stress again on the fact that even though the joint structure signature is involved in Eq.~\eqref{eq:sfd76} it has to be regarded as a combinatorial object (see Remark~\ref{rem:as5f}). Therefore in Proposition~\ref{prop:mainDEC9} we do not need to assume that the component lifetimes are i.i.d.\ or exchangeable.
\item[(b)] Moreover, it is well known that the decomposition \eqref{eq:837gf465} still holds in the case of exchangeable component lifetimes with possible ties. Considering the joint structure signature as a combinatorial object allows us to use it when $F$ has possible ties. It is then easy to check that Proposition \ref{prop:mainDEC9} still holds in this case.
\item[(c)] The proof of Proposition~\ref{prop:mainDEC9} mainly uses combinatorial arguments and therefore constitutes
an alternative proof of the decomposition reported in Navarro et al.~\cite{NavSamBal13}.
\end{enumerate}
\end{remark}

The following example shows that representation \eqref{eq:sfd76} cannot be extended to arbitrary joint distribution functions $F$, even for particular systems.

\begin{example}
Consider the systems $\S_1=(C,\phi_1,F)$ and $\S_2=(C,\phi_2,F)$, where $\phi_1(x_1,x_2)=x_1$ and $\phi_2(x_1,x_2)=x_1\wedge x_2$ and assume that the component lifetimes $T_1$ and $T_2$ are independent and have exponential distributions with reliability functions $\overline{F}_1(t)=e^{-t}$ and $\overline{F}_2(t)=e^{-2t}$, respectively. Then the joint reliability function of these systems is given by
$$
\overline{F}_{\S_1,\S_2}(t_1,t_2) ~=~
\begin{cases}
e^{-3t_2}, & \mbox{if $t_1\leqslant t_2$},\\
e^{-t_1-2t_2}, & \mbox{if $t_2\leqslant t_1$}.
\end{cases}
$$
We also have
\begin{eqnarray*}
\overline{F}_{1:2,1:2}(t_1,t_2) &=&
\begin{cases}
e^{-3t_2}, & \mbox{if $t_1\leqslant t_2$},\\
e^{-3t_1}, & \mbox{if $t_2\leqslant t_1$},
\end{cases}
\\
\overline{F}_{1:2,2:2}(t_1,t_2) &=&
\begin{cases}
e^{-2t_1-t_2}+e^{-t_1-2t_2}-e^{-3t_2}, & \mbox{if $t_1\leqslant t_2$},\\
e^{-3t_1}, & \mbox{if $t_2\leqslant t_1$},
\end{cases}
\\
\overline{F}_{2:2,1:2}(t_1,t_2) &=&
\begin{cases}
e^{-3t_2}, & \mbox{if $t_1\leqslant t_2$},\\
e^{-2t_1-t_2}+e^{-t_1-2t_2}-e^{-3t_1}, & \mbox{if $t_2\leqslant t_1$},
\end{cases}
\\
\overline{F}_{2:2,2:2}(t_1,t_2) &=&
\begin{cases}
e^{-t_2}+e^{-2t_2}-e^{-3t_2}, & \mbox{if $t_1\leqslant t_2$},\\
e^{-t_1}+e^{-2t_1}-e^{-3t_1}, & \mbox{if $t_2\leqslant t_1$}.
\end{cases}
\end{eqnarray*}
Since the exponential functions involved in the expressions above are linearly independent,
one can easily show that there cannot exist real numbers $w_{1,1}$, $w_{1,2}$, $w_{2,1}$, and $w_{2,2}$ such that
$$
\overline{F}_{\S_1,\S_2} ~=~ w_{1,1}\overline{F}_{1:2,1:2}+w_{1,2}\overline{F}_{1:2,2:2}+w_{2,1}\overline{F}_{2:2,1:2}+
w_{2,2}\overline{F}_{2:2,2:2}.
$$
In particular, representation \eqref{eq:sfd76} does not hold for these systems.
\end{example}

\begin{remark}
When \eqref{eq:sfd76} does not hold, proceeding as in \cite{NavSpiBal10} we can consider two pairs of bivariate
\emph{mixed systems} (i.e., mixtures of pairs of semicoherent systems with component lifetimes $T_1,\dots,T_n$)
associated with $(\S_1,\S_2)$. The first pair $(\widetilde{\mathcal{S}}_1,\widetilde{\mathcal{S}}_2)$ is
called \textit{bivariate projected systems} and defined to be equal to $(\S_{k:n},\S_{l:n})$ with probability $p_{k,l}$,
where $\S_{k:n}$ is the $k$-out-of-$n$ system. The second pair $(\overline{\mathcal{S}}_1,\overline{\mathcal{S}}_2)$ is
called \textit{bivariate average (or symmetrized) systems} and defined to be equal to $(\S_{k:n},\S_{l:n})$ with probability
$s_{k,l}$. These latter systems $(\overline{\mathcal{S}}_1,\overline{\mathcal{S}}_2)$ are equivalent (equal in law) to the
parent systems $(\S_1,\S_2)$ when all the components are placed randomly in the structures. Equivalently, they are the systems
obtained with the structure functions $\phi_1$ and $\phi_2$ from the exchangeable random vector $(T_1^{\text{EXC}},\dots,T_n^{\text{EXC}})$ defined to be equal to
$(T_{\sigma(1)},\dots, T_{\sigma(n)})$ with probability $1/n!$ for any permutation $\sigma$. Both pairs $(\widetilde{\mathcal{S}}_1,\widetilde{\mathcal{S}}_2)$ and $(\overline{\mathcal{S}}_1,\overline{\mathcal{S}}_2)$ can be seen as approximations (in law) of the pair $(\mathcal{S}_1,\mathcal{S}_2)$.
\end{remark}

It was observed in \cite[Remark~2]{MarMatWal11} that in general the exchangeability of the component states does not imply the exchangeability of the component lifetimes. We now show that Eq.~\eqref{eq:cond.12} provides an intermediate condition between these exchangeabilities. More precisely, we have the following result.

\begin{proposition}\label{prop123123}
The following assertions hold.
\begin{enumerate}
\item[(a)] If the component lifetimes $T_1,\ldots,T_n$ are exchangeable, then condition \eqref{eq:cond.12} holds for any $t_1,t_2\geqslant 0$ and any $\bfx,\bfy\in\{0,1\}^n$.

\item[(b)] If condition \eqref{eq:cond.12} holds for some $0\leqslant t_1<t_2$ and any nonzero $\bfx,\bfy\in\{0,1\}^n$, then the component states $X_1(t_2)$, $\ldots$, $X_n(t_2)$ at time $t_2$ are exchangeable.
\end{enumerate}
\end{proposition}

\begin{proof}
(a) Suppose without loss of generality that $t_1\leqslant t_2$. Then the event $(\mathbf{X}(t_1)=\bfx~\text{ and }~\mathbf{X}(t_2)=\bfy)$ is empty if $\bfy\nleqslant\bfx$. Otherwise, it is equal to
$$
\bigg(\bigcap_{x_i=0}(T_i\leqslant t_1)\bigg)\cap\bigg(\bigcap_{\textstyle{x_i=1\atop y_i=0}}(t_1<T_i\leqslant t_2)\bigg)\cap\bigg(\bigcap_{y_i=1}(t_2<T_i)\bigg).
$$
Since $T_1,\ldots,T_n$ are exchangeable, we see that the probability of this event remains unchanged if we replace $T_1,\ldots,T_n$ with $T_{\sigma(1)},\ldots,T_{\sigma(n)}$. This shows that condition \eqref{eq:cond.12} holds for any $\bfx,\bfy\in\{0,1\}^n$ and any permutation $\sigma\in \Sigma_n$.

(b) For every $\sigma\in \Sigma_n$, the equality $\Pr(\mathbf{X}(t_2)=\bfy)=\Pr(\mathbf{X}(t_2)=\sigma(\bfy))$ is trivial if $\bfy=\mathbf{0}$. Otherwise, since $t_1<t_2$, we obtain
\begin{eqnarray*}
\Pr(\mathbf{X}(t_2)=\bfy) &=& \sum_{\textstyle{\bfx\in\{0,1\}^n\atop\bfx\geqslant\bfy}}\Pr({\mathbf{X}(t_1)\choose \mathbf{X}(t_2)}={\bfx\choose\bfy})\\
&=& \sum_{\textstyle{\bfx\in\{0,1\}^n\atop\bfx\geqslant\bfy}}\Pr({\mathbf{X}(t_1)\choose \mathbf{X}(t_2)}={\sigma(\bfx)\choose\sigma(\bfy)}).
\end{eqnarray*}
Using the fact that $\bfx\geqslant\bfy$ if and only if $\sigma(\bfx)\geqslant\sigma(\bfy)$, we then have
\begin{eqnarray*}
\Pr(\mathbf{X}(t_2)=\bfy) &=& \sum_{\textstyle{\bfx\in\{0,1\}^n\atop\sigma(\bfx)\geqslant\sigma(\bfy)}}\Pr({\mathbf{X}(t_1)\choose \mathbf{X}(t_2)}={\sigma(\bfx)\choose\sigma(\bfy)})\\
&=& \Pr(\mathbf{X}(t_2)=\sigma(\bfy)).
\end{eqnarray*}
This completes the proof.
\end{proof}

In most of the previous results, we have considered fixed times $t_1,t_2$. If we are interested in condition \eqref{eq:cond.12} for any times $t_1,t_2$, then we have the following result.

\begin{proposition}
Condition \eqref{eq:cond.12} holds for any $t_1,t_2>0$ and any nonzero $\bfx,\bfy\in\{0,1\}^n$ if and only if it holds for any $t_1,t_2>0$ and any $\bfx,\bfy\in\{0,1\}^n$. Moreover in this case, the component states $X_1(t)$, $\ldots$, $X_n(t)$ are exchangeable at any time $t>0$.
\end{proposition}

\begin{proof}
Fix a time $t>0$ and choose $t_1$ such that $0<t_1<t$. Since condition \eqref{eq:cond.12} holds for $t_1$ and $t$ and any nonzero $\bfx,\bfy\in\{0,1\}^n$, the component states $X_1(t)$, $\ldots$, $X_n(t)$ are exchangeable by Proposition~\ref{prop123123}.

Now, assume without loss of generality that $0<t_1\leqslant t_2$ and compute, for $\bfx\neq\mathbf{0}$,
\begin{multline*}
\Pr({\mathbf{X}(t_1)\choose \mathbf{X}(t_2)}={\bfx\choose\mathbf{0}})\\
~=~ \Pr(\mathbf{X}(t_1)=\bfx)-\sum_{\textstyle{\bfy\in\{0,1\}^n\setminus\{\mathbf{0}\}\atop\bfy\leqslant\bfx}}\Pr({\mathbf{X}(t_1)\choose \mathbf{X}(t_2)}={\bfx\choose\bfy}).
\end{multline*}
Due to the exchangeability of the component states at times $t_1$ and to the hypothesis, this expression remains unchanged if we replace $\bfx$ with $\sigma(\bfx)$ for any $\sigma\in \Sigma_n$.
\end{proof}

We end this section by stating the generalization of Proposition~\ref{prop:T2} to the case of two systems. The proof is a straightforward adaptation of the proof of Proposition~\ref{prop:mainDEC9}. However, just as in the case of a single system (see \cite[Theorem~6]{MarMatWal11}), the corresponding condition cannot be easily interpreted.

\begin{proposition}\label{prop:mainDEC9z}
Let $t_1,t_2\geqslant 0$ be fixed. If the joint c.d.f.\ $F$ satisfies the condition
\begin{multline}
\Pr({\mathbf{X}(t_1)\choose \mathbf{X}(t_2)}={\mathbf{x}\choose\mathbf{y}})\\
~=~ q(\bfx,\bfy)\sum_{\textstyle{\mathbf{u}\in\{0,1\}^n\atop |\mathbf{u}|=|\bfx|}}\sum_{\textstyle{\mathbf{v}\in\{0,1\}^n\atop |\mathbf{v}|=|\bfy|}}\Pr({\mathbf{X}(t_1)\choose \mathbf{X}(t_2)}={\mathbf{u}\choose\mathbf{v}})\label{eq:weird}
\end{multline}
for any nonzero
$\bfx,\bfy\in\{0,1\}^n$, then we have
\begin{equation}\label{eq:111222333444}
\overline{F}_{\mathcal{S}_1,\mathcal{S}_2}(t_1,t_2) ~=~ \sum_{k=1}^n{\,}\sum_{l=1}^n p_{k,l}{\,}\overline{F}_{k:n,l:n}(t_1,t_2),
\end{equation}
for any semicoherent systems $\mathcal{S}_1,\mathcal{S}_2$.
Conversely, if $n\geqslant 3$ and if \eqref{eq:111222333444} holds for any coherent systems $\mathcal{S}_1,\mathcal{S}_2$ (at times $t_1,t_2$), then
the joint c.d.f.\ $F$ satisfies condition \eqref{eq:weird} for any nonzero $\bfx,\bfy\in\{0,1\}^n$.
\end{proposition}

\section{Applications to multistate systems}

We now apply the results obtained in Sections 2 and 3 to the investigation of the signature and reliability of multistate systems made up of two-state components.

Let $n\geqslant 1$ and $m\geqslant 1$ be fixed integers. Consider an $(m+1)$-state system $\calS=(C,\phi,F)$, where $C=[n]$ represents $n$ nonrepairable two-state components, $\phi\colon\{0,1\}^n\to\{0,\ldots,m\}$ is the structure function that expresses the state of the system in terms of the states of its components, and as usual $F$ denotes the joint c.d.f.\ of the component lifetimes $T_1,\ldots,T_n$.

Note that the structure function $\phi$ has now $m+1$ possible values that represent the possible states of the system, going from the lowest value ``$0$'' representing the \emph{complete failure state} to the highest value ``$m$'' representing the \emph{perfection state}. Such multistate systems made up of two-state components have been investigated and applied to network reliability for instance in Levitin et al. \cite{LevGerShp11}, Gertsbakh et al.~\cite{GerSchSpi12}, and Da and Hu \cite{DaHu13}.

We assume that the system $\calS$ (or equivalently, its structure function $\phi$) is \emph{semicoherent}, i.e., it is nondecreasing in each variable and satisfies the boundary conditions $\phi(\mathbf{0})=0$ and $\phi(\mathbf{1})=m$. Here again, we also assume that the c.d.f.\ $F$ has no ties.

Recall that $X_j(t)=\Ind(T_j>t)$ denotes the state variable of component $j\in C$ at time $t\geqslant 0$, and $\mathbf{X}(t)$ denotes the $n$-tuple $(X_1(t),\ldots,X_n(t))$. By definition, $X_{\calS}(t)=\phi(\mathbf{X}(t))$ is then precisely the state of the system at time $t$.

Since the system has $m+1$ possible states, its ``lifetime'' can be described by $m$ random variables that represent the times at which the state of the system strictly decreases. We introduce these random variables, denoted $T_{\calS}^{\geqslant 1},\ldots,T_{\calS}^{\geqslant m}$, by means of the conditions
$$
T_{\calS}^{\geqslant k}~>~t \quad \Leftrightarrow \quad \phi(\mathbf{X}(t))~\geqslant ~k{\,},\qquad k=1,\ldots,m.
$$
Thus defined, $T_{\calS}^{\geqslant k}$ is the time at which the system ceases to be at a state $\geqslant k$ and deteriorates to a state $<k$. It is then clear that these variables satisfy the inequalities $T_{\calS}^{\geqslant 1}\geqslant\cdots\geqslant T_{\calS}^{\geqslant m}$.

In this setting it is also useful and natural to introduce the reliability function (called reliability of the system at states $\geqslant k$)
\begin{equation}\label{eq:pRelF2}
\overline{F}_{\calS}^{\geqslant k}(t) ~=~ \Pr(T_{\calS}^{\geqslant k}>t),\qquad t\geqslant 0,
\end{equation}
for $k=1,\ldots,m$ as well as the (overall) reliability function
\begin{multline}
\overline{F}_{\calS}(t_1,\ldots,t_m)\\
~=~ \Pr(T_{\calS}^{\geqslant 1}>t_1~,\ldots,~T_{\calS}^{\geqslant m}>t_m),\quad t_1,\ldots,t_m\geqslant 0.\label{eq:oRelF3}
\end{multline}
Due to the inequalities $T_{\calS}^{\geqslant 1}\geqslant\cdots\geqslant T_{\calS}^{\geqslant m}$, it is clear that we also have $\overline{F}_{\calS}^{\geqslant 1}\geqslant\cdots\geqslant \overline{F}_{\calS}^{\geqslant m}$.

For simplicity let us now consider the special case where $m=2$. In this setting, the {\em probability signature} of a $3$-state system $\calS=(C,\phi,F)$ is the square matrix $\boldsymbol{\mathfrak{p}}$ of order $n$ whose $(k,l)$-entry is the probability
\begin{equation}\label{eq:sign3}
\mathfrak{p}_{k,l} ~=~ \Pr(T_{\calS}^{\geqslant 1}=T_{k:n}~,~T_{\calS}^{\geqslant 2}=T_{l:n}),\quad k,l=1,\ldots,n,
\end{equation}
where $T_{1:n}\leqslant\cdots\leqslant T_{n:n}$ are the order statistics of the component lifetimes $T_1,\ldots,T_n$. Note that if $\mathfrak{p}_{k,l}> 0$, then necessarily $k\geqslant l$.

Also, the {\em tail probability signature} of a $3$-state system $\calS=(C,\phi,F)$ is the square matrix $\overline{\boldsymbol{\mathfrak{P}}}$ of order $n+1$ whose $(k,l)$-entry is the probability
\begin{equation}\label{eq:tsign3}
\overline{\mathfrak{P}}_{k,l} ~=~ \Pr(T_{\calS}^{\geqslant 1}>T_{k:n}~,~T_{\calS}^{\geqslant 2}>T_{l:n}),\quad k,l=0,\ldots,n.
\end{equation}
We remark that the concepts of probability signature and tail probability signature were already introduced by Gertsbakh et al.~\cite{GerSchSpi12} and later by Da and Hu \cite{DaHu13} as the ``bivariate signature'' and ``bivariate tail signature'', respectively, of a $3$-state system in the special case where the component lifetimes are i.i.d.\ and continuous (see also \cite{LevGerShp11} for an earlier work). In the general non-i.i.d.\ setting, we naturally call these concepts ``probability signature'' and ``tail probability signature'' in full analogy with the case of two-state systems (see Section 2). Clearly, these definitions can be immediately extended to the more general case of $(m+1)$-state systems.

In Sections 5 and 6 we will show how the results obtained in Sections 2 and 3 on the investigation of the joint probability signature and joint reliability of two or more systems can be usefully and very easily applied to the computation of the probability signature and reliability function of a multistate system made up of two-state components. Our results actually simplify and generalize to the non-i.i.d.\ case several results obtained by Gertsbakh et al.~\cite{GerSchSpi12} and Da and Hu \cite{DaHu13} on these topics. Moreover, we do not require the restrictive ``regularity condition'' (which states that the system must be coherent and that the state of the system may not suddenly decrease by more than one unit; see Remark~\ref{rem:reg5} below).

The key feature of our approach is given by the Decomposition Principle given in Section 5, which states that such a multistate system can always be additively decomposed into several two-state systems. Although this feature was already observed in the eighties by Block and Savits \cite[Theorem 2.8]{BloSav82}, it seems that it has never been exploited to investigate the probability signature of a multistate system.



\section{Decomposition principle and some of its consequences}

By defining the state variable $X_{\calS}^{\geqslant k}(t)= \Ind(T_{\calS}^{\geqslant k}>t)$ for $k=1,\ldots,m$, we can immediately express the state $\phi(\mathbf{X}(t))$ of the system at time $t$ by
$$
X_{\calS}(t) ~=~ \sum_{k=1}^m X_{\calS}^{\geqslant k}(t).
$$
This equation shows that it is natural to express the structure function $\phi$ as a sum of $m$ Boolean (i.e., $\{0,1\}$-valued) structure functions. Actually, this idea was already suggested in another form by Block and Savits \cite[Theorem 2.8]{BloSav82}.

\begin{proposition}[Boolean decomposition]\label{prop:DecBloSav}
Any semicoherent structure function $\phi\colon\{0,1\}^n\to\{0,\ldots,m\}$ decomposes in a unique way as a sum
\begin{equation}\label{eq:DecBloSav3}
\phi ~=~ \sum_{k=1}^m\phi_{\ang{k}},
\end{equation}
where $\phi_{\ang{k}}\colon\{0,1\}^n\to\{0,1\}$ $(k=1,\ldots,m)$ are semicoherent structure functions such that $\phi_{\ang{1}}\geqslant\cdots\geqslant\phi_{\ang{m}}$ (i.e., $\phi_{\ang{1}}(\bfx)\geqslant\cdots\geqslant\phi_{\ang{m}}(\bfx)$ for all $\bfx\in\{0,1\}^n$).
\end{proposition}

\begin{proof}
(Uniqueness) Assume that we have functions $\phi_{\ang{1}}\geqslant\cdots\geqslant\phi_{\ang{m}}$ for which \eqref{eq:DecBloSav3} holds. For any $j\in\{1,\ldots,m\}$ and any $\bfx\in\{0,1\}^n$, if $\phi_{\ang{j}}(\bfx)=1$, then $\phi_{\ang{i}}(\bfx)=1$ for every $i\leqslant j$. By \eqref{eq:DecBloSav3} we then have $\phi(\bfx)\geqslant j$. Similarly, if $\phi_{\ang{j}}(\bfx)=0$, then $\phi(\bfx)< j$. Therefore, for every $k\in\{0,\ldots,m\}$ we must have
\begin{equation}\label{eq:ConstrPk8}
\phi_{\ang{k}}(\bfx) ~=~ 1\quad\Leftrightarrow\quad\phi(\bfx) ~\geqslant ~ k.
\end{equation}

\noindent (Existence) It is straightforward to see that the decomposition holds if we define each $\phi_{\ang{k}}\colon\{0,1\}^n\to\{0,1\}$ by the condition \eqref{eq:ConstrPk8}. Each of these functions is clearly semicoherent.
\end{proof}

\begin{remark}
We observe that the equivalence \eqref{eq:ConstrPk8} strongly resembles condition (4.1) in Natvig \cite{Nat82} (see also \cite[Def.~2.6, p.~16]{Nat11}). Actually, taking into consideration Section 6 of \cite{Nat82}, the equivalence \eqref{eq:ConstrPk8} shows that every multistate system made up of two-state components is a (modified) multistate coherent system of type 2 in Natvig's sense.
\end{remark}


The Boolean decomposition stated in Proposition~\ref{prop:DecBloSav} shows that any semicoherent $(m+1)$-state system always gives rise to $m$ semicoherent two-state systems constructed on the same set of components. This observation naturally leads to the following definition.

\begin{definition}\label{def:DecSyst4}
Given a semicoherent $(m+1)$-state system $\calS=(C,\phi,F)$, with Boolean decomposition $\phi=\sum_k\phi_{\ang{k}}$, we define the semicoherent systems $\calS_k=(C,\phi_{\ang{k}},F)$ for $k=1,\ldots,m$.
\end{definition}

From Definition~\ref{def:DecSyst4} we can immediately derive the following important theorem. For every $k\in\{1,\ldots,m\}$, let $T_{\calS_k}$ denote the random lifetime of $\calS_k$.

\begin{theorem}\label{thm:DecSyst4}
We have $T_{\calS}^{\geqslant k}=T_{S_k}$ for $k=1,\ldots,m$.
\end{theorem}

\begin{proof}
For any $k\in\{1,\ldots,m\}$ and any $t\geqslant 0$, by using \eqref{eq:ConstrPk8} we obtain the equivalence
\begin{eqnarray*}
T_{\calS}^{\geqslant k} ~>~t &\Leftrightarrow & \phi(\mathbf{X}(t))~\geqslant ~ k\quad\Leftrightarrow \quad \phi_{\ang{k}}(\mathbf{X}(t))~=~1\\
& \Leftrightarrow & T_{S_k}~>~t,
\end{eqnarray*}
which proves the theorem.
\end{proof}

As we will see in the rest of this paper, Definition~\ref{def:DecSyst4} and Theorem~\ref{thm:DecSyst4} have important consequences. They actually reveal the following decomposition principle, which shows that considering a multistate system reduces to considering simultaneously several two-state systems.
\begin{DP}
Any semicoherent $(m+1)$-state system $\calS$ made up of two-state components can be additively decomposed into $m$ semicoherent two-state systems $\calS_1,\ldots,\calS_m$ constructed on the same set of components, with the property that for any $k\in\{1,\ldots,m\}$ the lifetime of $\calS_k$ is the time at which $\calS$ deteriorates from a state $\geqslant k$ to a state $<k$.
\end{DP}

In Sections 2 and 3 the joint signature and the joint reliability of two or more systems have been investigated in detail and several explicit formulas for these concepts have been presented. As we will now see, the Decomposition Principle will enable us to use these results to efficiently investigate both the signature and reliability of multistate systems made up of two-state components. For simplicity, we will restrict ourselves to the case $m=2$.

Considering the systems $\calS_1$ and $\calS_2$ introduced in Definition~\ref{def:DecSyst4} and using Theorem~\ref{thm:DecSyst4}, we immediately see that the probability signature and the tail probability signature of any $3$-state system (as defined in \eqref{eq:sign3} and \eqref{eq:tsign3}) are precisely the joint probability signature and the joint tail probability signature, respectively, of the systems $\calS_1$ and $\calS_2$. Thus, we have established the following proposition.

\begin{proposition}\label{prop:pp2}
We have $\boldsymbol{\mathfrak{p}}=\mathbf{p}$ and $\overline{\boldsymbol{\mathfrak{P}}}=\overline{\mathbf{P}}$.
\end{proposition}

Proposition~\ref{prop:pp2} not only provides an interesting connection with the results presented in Sections 2 and 3 but also provides an explicit expression (see Theorem~\ref{thm:main32} below) of the tail probability signature of a multistate system in a way that simplifies and generalizes to the non-i.i.d.\ case the results recently obtained by Gertsbakh et al.~\cite{GerSchSpi12} and Da and Hu \cite{DaHu13} on these topics. From now on, we will then use the symbols $\mathbf{p}$ and $\overline{\mathbf{P}}$ in place of $\boldsymbol{\mathfrak{p}}$ and $\overline{\boldsymbol{\mathfrak{P}}}$, respectively.

We first observe that Proposition~\ref{prop:ConvForm5873} immediately provides the conversion formulas between the matrices $\mathbf{p}$ and $\overline{\mathbf{P}}$ (already obtained in the i.i.d.\ case in \cite{GerSchSpi12}).

In this setting of multistate systems we have already observed the following condition on the matrix $\mathbf{p}$: if $p_{k,l}> 0$, then $k\geqslant l$. From this observation we easily derive the following condition on the matrix $\overline{\mathbf{P}}$ (obtained in another form and in the i.i.d.\ case by Da and Hu~\cite[p.~148]{DaHu13}).

\begin{proposition}\label{prop:kl3}
If $k\leqslant l$, then we have $\overline{P}_{k,l}=\overline{P}_{l,l}$. That is, each value $\overline{P}_{k,l}$ located in the column $l$ and above the diagonal of $\overline{\mathbf{P}}$ is precisely given by the corresponding value $\overline{P}_{l,l}$ on the diagonal.
\end{proposition}

\begin{proof}
Since $p_{i,j}=0$ if $i<j$, by Eq.~\eqref{eq:CoFo1} for $k\leqslant l$ we have $\overline{P}_{k,l} = \sum_{i\geqslant j>l}p_{i,j} = \overline{P}_{l,l}$.
\end{proof}

Now, from Theorem~\ref{thm:main3} we immediately derive the following theorem, which provides an explicit expression of $\overline{P}_{k,l}$ in which the structure function $\phi$ is represented through the functions $\phi_{\ang{1}}$ and $\phi_{\ang{2}}$ and the distribution $F$ of the component lifetimes is encoded in the joint relative quality function.

\begin{theorem}\label{thm:main32}
For every $k,l\in\{0,\ldots,n\}$ we have
$$
\overline{P}_{k,l} ~=~ \sum_{\textstyle{A\subseteq C\atop |A|=n-k}}{\,}\sum_{\textstyle{B\subseteq C\atop |B|=n-l}}q(A,B){\,}\phi_{\ang{1}}(A)\phi_{\ang{2}}(B).
$$
\end{theorem}

\begin{remark}\label{rem:skl1}
Recall that when the component lifetimes $T_1,\ldots,T_n$ are exchangeable and continuous (and in particular when they are i.i.d.\ and continuous), then $q(A,B)$ reduces to the value $q_0(A,B)$ as given in Definition~\ref{de:q0}. In this case the signature no longer depends on the distribution of the component lifetimes and can be seen as a purely combinatorial object that depends only on the structure function. This object is called \emph{structure signature} and the corresponding tail probability signature is called \emph{tail structure signature}. In accordance with the notation used in the previous sections, we then use the symbols $\mathbf{s}$, $\overline{\mathbf{S}}$, $s_{k,l}$, and $\overline{S}_{k,l}$ instead of $\mathbf{p}$, $\overline{\mathbf{P}}$, $p_{k,l}$, and $\overline{P}_{k,l}$, respectively.
\end{remark}

Let us now investigate the reliability functions $\overline{F}_{\calS}^{\geqslant 1}(t)$,  $\overline{F}_{\calS}^{\geqslant 2}(t)$, and $\overline{F}_{\calS}(t_1,t_2)$ as defined in \eqref{eq:pRelF2} and \eqref{eq:oRelF3}. Considering again the systems $\calS_1$ and $\calS_2$ introduced in Definition~\ref{def:DecSyst4}, we can introduce the reliability functions $\overline{F}_{S_1}(t)=\Pr(T_{\calS_1}>t)$ and $\overline{F}_{S_2}(t)=\Pr(T_{\calS_2}>t)$ of systems $\calS_1$ and $\calS_1$, respectively, as well as the joint reliability function
$$
\overline{F}_{\calS_1,\calS_2}(t_1,t_2) ~=~ \Pr(T_{\calS_1}>t_1~,~T_{\calS_2}>t_2).
$$

By using Theorem~\ref{thm:DecSyst4} we immediately see that the reliability functions $\overline{F}_{\calS}^{\geqslant 1}$ and $\overline{F}_{\calS}^{\geqslant 2}$ coincide with the reliability functions of the systems $\calS_1$ and $\calS_2$, that is,
$$
\overline{F}_{\calS}^{\geqslant 1} ~=~\overline{F}_{S_1} {\,},\quad\overline{F}_{\calS}^{\geqslant 2}~=~\overline{F}_{S_2}{\,}.
$$
Similarly, the reliability function $\overline{F}_{\calS}$ coincides with the joint reliability function of the systems $\calS_1$ and $\calS_2$, that is
$$
\overline{F}_{\calS}~=~\overline{F}_{\calS_1,\calS_2}.
$$

From this observation we immediately obtain
\begin{equation}\label{eq:OvRel42}
\overline{F}_{\calS}(t_1,t_2) ~=~ \Pr\big(\phi_{\ang{1}}(\mathbf{X}(t_1))=1~,~\phi_{\ang{2}}(\mathbf{X}(t_2))=1\big),
\end{equation}
from which we also derive the formula (see Proposition~\ref{prop:f5fd})
\begin{multline*}
\overline{F}_{\calS}(t_1,t_2)\\
 ~= \sum_{\bfx,\bfy\in\{0,1\}^n}{\,}\phi_{\ang{1}}(\bfx){\,}\phi_{\ang{2}}(\bfy){\,}\Pr({\mathbf{X}(t_1)\choose \mathbf{X}(t_2)}={\mathbf{x}\choose\mathbf{y}}).
\end{multline*}

Also, from Proposition~\ref{prop:mainDEC9} we derive the following proposition, which provides a sufficient condition on the component states for the reliability of system $\calS$ to admit a signature-based decomposition. Once again, this result simplifies and generalizes to the non-i.i.d.\ setting a recent result obtained by Gertsbakh et al.~\cite[Theorem 1]{GerSchSpi12} and Da and Hu \cite[Theorem 7.2.3]{DaHu13}.


\begin{proposition}\label{prop:mainDEC92}
If, for any $t_1,t_2\geqslant 0$, the joint c.d.f.\ $F$ satisfies condition \eqref{eq:cond.12} for any nonzero $\bfx,\bfy\in\{0,1\}^n$ and any permutation $\sigma$ on $[n]$, then we have
\begin{equation}\label{eq:sfd762}
\overline{F}_{\calS}(t_1,t_2) ~=~ \sum_{k=1}^n{\,}\sum_{l=1}^n s_{k,l}{\,}\overline{F}_{k:n,l:n}(t_1,t_2),
\end{equation}
where the coefficients $s_{k,l}$ correspond to the structure signature as mentioned in Remark~\ref{rem:skl1}.
\end{proposition}

\begin{remark}
We would like to stress on the fact that even though the structure signature is involved in Eq.~\eqref{eq:sfd762} it has to be regarded as a combinatorial object as mentioned in Remark~\ref{rem:skl1}. Therefore we do not need to assume in Proposition~\ref{prop:mainDEC92} that the component lifetimes are i.i.d.\ or exchangeable.
\end{remark}

Before closing this section, let us discuss a little further the Boolean decomposition stated in Proposition~\ref{prop:DecBloSav} to help the reader get familiar with this important technique.

The following immediate proposition can be seen as a converse result to Proposition~\ref{prop:DecBloSav}.

\begin{proposition}\label{prop:we75}
Given $m$ semicoherent structure functions $\phi_k\colon\{0,1\}^n\to\{0,1\}$ $(k=1,\ldots,m)$, the function $\phi=\sum_{k=1}^m\phi_k$ is a $\{0,\ldots,m\}$-valued semicoherent structure function.
\end{proposition}

By definition, if the functions $\phi_k$ given in Proposition~\ref{prop:we75} are not ordered as $\phi_1\geqslant\cdots\geqslant\phi_m$, then they do not correspond to the Boolean decomposition of $\phi$ as defined in Proposition~\ref{prop:DecBloSav}.

\begin{example}\label{ex:notComp4}
Consider the functions $\phi_1\colon\{0,1\}^3\to\{0,1\}$ and $\phi_2\colon\{0,1\}^3\to\{0,1\}$ defined by
$$
\phi_1(x_1,x_2,x_3) ~=~ x_1\wedge(x_2\vee x_3)
$$
and
$$
\phi_2(x_1,x_2,x_3) ~=~ x_2\wedge x_3.
$$
We have $\phi_1(1,0,1)>\phi_2(1,0,1)$ and $\phi_1(0,1,1)<\phi_2(0,1,1)$. Hence $\phi_1$ and $\phi_2$ are not comparable (i.e., $\phi_1\nleqslant\phi_2$ and $\phi_2\nleqslant\phi_1$). Therefore, they do not correspond to the Boolean decomposition of $\phi=\phi_1+\phi_2$.
\end{example}

\begin{proposition}\label{prop:d8f76}
If $\phi=\phi_1+\phi_2$ for some functions $\phi_k\colon\{0,1\}^n\to\{0,1\}$ $(k=1,2)$, then  $\phi_{\ang{1}}=\phi_1\vee\phi_2$ and $\phi_{\ang{2}}=\phi_1\wedge\phi_2$. More generally, if $\phi=\phi_1+\cdots +\phi_m$ for some functions $\phi_k\colon\{0,1\}^n\to\{0,1\}$ $(k=1,\ldots,m)$, then
$$
\phi_{\ang{k}}=\bigvee_{\textstyle{A\subseteq\{1,\ldots,m\}\atop |A|=k}}~\bigwedge_{i\in A}\phi_i{\,}.
$$
In other terms, for every $\bfx\in\{0,1\}^m$, the value of $\phi_{\ang{k}}(\bfx)$ is the $k$th largest value from among $\phi_1(\bfx),\ldots,\phi_m(\bfx)$.
\end{proposition}

\begin{proof}
Let $\bfx\in\{0,1\}^n$. By \eqref{eq:ConstrPk8} we have $\phi_{\ang{k}}(\bfx)=1$ if and only if $\sum_{l=1}^m\phi_l(\bfx)\geqslant k$. This means that at least $k$ of the values $\phi_1(\bfx),\ldots,\phi_m(\bfx)$ are equal to $1$, or equivalently, the $k$th largest of these values is equal to $1$.
\end{proof}

\begin{example}[Example~\ref{ex:notComp4} continued]\label{ex:notComp42}
Consider the function $\phi=\phi_1+\phi_2$, where $\phi_1$ and $\phi_2$ are given in Example~\ref{ex:notComp4}. Then, by Proposition~\ref{prop:d8f76}, we have
\begin{eqnarray*}
\phi_{\ang{1}}(\bfx) &=& \phi_1(\bfx)\vee \phi_2(\bfx)\\
&=& (x_1\wedge x_2)\vee(x_2\wedge x_3)\vee(x_3\wedge x_1),\\
\phi_{\ang{2}}(\bfx) &=& \phi_1(\bfx)\wedge \phi_2(\bfx) ~=~ x_1\wedge x_2\wedge x_3{\,},
\end{eqnarray*}
and the Boolean decomposition of $\phi$ is $\phi=\phi_{\ang{1}}+\phi_{\ang{2}}$.
\end{example}

\begin{remark}
The functions $\phi_{\ang{1}},\ldots,\phi_{\ang{m}}$ can also be defined in terms of ``path sets'' in a very easy way. For any $k\in\{1,\ldots,m\}$ we say that a subset $P$ of components is a \emph{$k$-path set} of the system $\calS=(C,\phi,F)$ if $\phi(P)\geqslant k$. We say that a $k$-path set is \emph{minimal} if it does not contain any other $k$-path set. Using \eqref{eq:ConstrPk8} we immediately see that $P$ is a (minimal) $k$-path set of $\calS$ if and only if it is a (minimal) path set of $\calS_k=(C,\phi_{\ang{k}},F)$. The system $\calS_k$ can therefore be defined from the $k$-path sets of $\calS$ for $k=1,\ldots,m$. Considering Example~\ref{ex:notComp4} for instance, we obtain the minimal $1$-path sets $\{1,2\}$, $\{1,3\}$, $\{2,3\}$, and the minimal $2$-path set $\{1,2,3\}$. This leads to the structure functions $\phi_{\ang{1}}$ and $\phi_{\ang{2}}$ as described in Example~\ref{ex:notComp42}.
\end{remark}

\begin{remark}\label{rem:reg5}
In the literature on multistate systems, it is often assumed that the system be coherent (i.e., in addition to be semicoherent, the structure function $\phi$ is nonconstant in each of its variables) and that the state of the system may not suddenly decrease by more than one unit. This restriction is called ``regularity condition'' in, e.g., \cite{DaHu13,GerShp11,GerSchSpi12}. In the present paper we do not need this restriction. On the one hand, being ``coherent'' is a condition that is not easy to handle and that is rarely required from a mathematical viewpoint. In some cases, this condition could even become a major obstacle. For instance, when decomposing a $\{0,\ldots,m\}$-valued structure function $\phi$ into a sum of $\{0,1\}$-valued structure functions $\phi_1,\ldots,\phi_m$ as in Example~\ref{ex:notComp4}, it would seem very restrictive to consider only coherent structure functions for $\phi_1,\ldots,\phi_m$. We therefore only ask the systems to be semicoherent. On the other hand, we allow ourselves that for any $k\in\{1,\ldots,m-1\}$ the event $T_{\calS}^{\geqslant k}=T_{\calS}^{\geqslant k+1}$ has a nonzero probability. Forbidding this condition would be very unnatural. For instance, consider the simple $3$-state system on $C=\{1,2,3\}$ whose structure function is given by $\phi=\phi_{\ang{1}}+\phi_{\ang{2}}$, with
$$
\phi_{\ang{1}}(x_1,x_2,x_3) ~=~ x_1\wedge (x_2\vee x_3)
$$
and
$$
\phi_{\ang{2}}(x_1,x_2,x_3) ~=~ x_1\wedge x_2\wedge x_3.
$$
In this case we have $\phi(1,1,1)=2$ and $\phi(0,1,1)=0$. Such a situation is not unnatural and hence should not be discarded.
\end{remark}

\section{Some applications}

We now consider some examples and applications of the theory developed in Section 5, which demonstrate the power of the Decomposition Principle when compared with previous methods.

\subsection{Some illustrative examples}

Let us consider a couple of special applications and examples to illustrate our results.

\begin{example}\label{ex:1c1}
Let us compute the reliability and structure signature of the system considered in Example~\ref{ex:notComp42}. By using \eqref{eq:f5sffx} and \eqref{eq:OvRel42} we immediately obtain
$$
\overline{F}_{\calS}(t_1,t_2) ~=~ \Pr\big(|\mathbf{X}(t_1)|\geqslant 2~,~ \mathbf{X}(t_2)=(1,1,1)\big){\,},\quad t_1,t_2\geqslant 0,
$$
and
$$
\overline{P}_{k,l} ~=~ \sum_{\textstyle{|A|\geqslant 2\atop |A|=3-k}}{\,}\sum_{\textstyle{|B|=3\atop |B|=3-l}}q(A,B){\,},\qquad k,l=0,\ldots,n.
$$

We know that $\overline{P}_{0,0}=1$ and we also see that $\overline{P}_{1,0} = \sum_{|A|=2} q(A,C) = q(C) = 1$. Using also \eqref{eq:CoFo2} we see that the matrices $\overline{\mathbf{P}}$ and $\mathbf{p}$ are then given by
$$
\overline{\mathbf{P}} ~=~
\begin{bmatrix}
1 & 0 & 0 & 0\\
1 & 0 & 0 & 0\\
0 & 0 & 0 & 0\\
0 & 0 & 0 & 0
\end{bmatrix}
\quad\mbox{and}\quad
\mathbf{p} ~=~
\begin{bmatrix}
0 & 0 & 0\\
1 & 0 & 0\\
0 & 0 & 0
\end{bmatrix}
.
$$
Note that the sparsity of these matrices comes from the fact that both $\phi_{\ang{1}}$ and $\phi_{\ang{2}}$ are symmetric functions (i.e., they represent $k$-out-of-$n$ systems).
\end{example}

\begin{example}
Let $\calS_1=(C,\phi_1,F)$ and $\calS_2=(C,\phi_2,F)$ be the systems considered in Example~\ref{ex:1}. Let also $\calS=(C,\phi,F)$ be the multistate system whose structure function is given by $\phi=\phi_1+\phi_2$.

Let us compute the probability signature of $\calS$. First, it is easy to see that the Boolean decomposition of $\phi$ is given by $\phi=\phi_{\ang{1}}+\phi_{\ang{2}}$, where
$$
\phi_{\ang{1}}(\bfx) ~=~ (x_1\wedge x_2)\vee(x_2\wedge x_4)\vee(x_3\wedge x_4)
$$
and
$$
\phi_{\ang{2}}(\bfx) ~=~ x_1\wedge x_2\wedge x_4.
$$
We can now compute the matrices $\overline{\mathbf{P}}$ and $\mathbf{p}$ by using \eqref{eq:CoFo2} and \eqref{eq:f5sffx}. Proceeding as in Example~\ref{ex:1c1} we easily obtain
$$
\overline{\mathbf{P}} ~=~
\begin{bmatrix}
1 & \overline{P}_{0,1} & 0 & 0 & 0\\
1 & \overline{P}_{1,1} & 0 & 0 & 0\\
\overline{P}_{2,0} & \overline{P}_{2,1} & 0 & 0 & 0\\
0 & 0 & 0 & 0 & 0\\
0 & 0 & 0 & 0 & 0
\end{bmatrix}
$$
and
$$
\mathbf{p} ~=~
\begin{bmatrix}
0 & 0 & 0 & 0\\
p_{2,1} & p_{2,2} & 0 & 0\\
p_{3,1} & p_{3,2} & 0 & 0\\
0 & 0 & 0 & 0
\end{bmatrix}
,
$$
where
\begin{eqnarray*}
\overline{P}_{2,0} &=& q(\{1,2\})+q(\{2,4\})+q(\{3,4\})\\
\overline{P}_{0,1} &=& \overline{P}_{1,1} ~=~ q(\{1,2,4\})\\
\overline{P}_{2,1} &=& q(\{1,2,4\},\{1,2\})+q(\{1,2,4\},\{2,4\})
\end{eqnarray*}
and
\begin{eqnarray*}
p_{2,1} &=& 1+\overline{P}_{2,1}-\overline{P}_{1,1}-\overline{P}_{2,0}\\
p_{2,2} &=& \overline{P}_{1,1}-\overline{P}_{2,1}\\
p_{3,1} &=& \overline{P}_{2,0}-\overline{P}_{2,1}\\
p_{3,2} &=& \overline{P}_{2,1}.
\end{eqnarray*}
\end{example}

\subsection{Probability that the system be at a given state}

It is noteworthy that from the Decomposition Principle and in particular from Eq.~\eqref{eq:ConstrPk8}, we immediately obtain the formula
\begin{multline*}
\Pr(\phi(\mathbf{X}(t))\geqslant k) ~=~ \Pr(\phi_{\ang{k}}(\mathbf{X}(t))=1) ~=~ \overline{F}_{\calS}^{\geqslant k}(t)\\
~=~ \overline{F}_{\calS_k}(t){\,},\qquad k\in\{1,\ldots,m\}.
\end{multline*}
That is, the probability that the system be at least at state $k$ at time $t\geqslant 0$ is simply given by the reliability $\overline{F}_{\calS_k}(t)$.

From this observation, we immediately derive the following explicit expression of the probability that the system be exactly at state $k$ at time $t\geqslant 0$. We remark that a similar result was established in the i.i.d.\ case in \cite[Theorem 2]{GerShp11}.

\begin{proposition}
We have
$$
\Pr(\phi(\mathbf{X}(t))=k) ~=~ \overline{F}_{\calS_k}(t)-\overline{F}_{\calS_{k+1}}(t),
$$
where by convention $\overline{F}_{\calS_{m+1}}(t)=0$ (i.e., $\phi_{\ang{m+1}}\equiv 0$).
\end{proposition}

\subsection{Dual systems}

Recall that the \emph{dual} of a structure function $\phi\colon\{0,1\}^n\to\{0,1\}$ is the structure function $\phi^D\colon\{0,1\}^n\to\{0,1\}$ defined from $\phi$ by
$$
\phi^D(\bfx) ~=~ 1-\phi(\mathbf{1}-\bfx){\,},\qquad\bfx\in\{0,1\}^n,
$$
where $\mathbf{1}$ denotes the $n$-tuple $(1,\ldots,1)$. Similarly, for multistate systems El-Neweihi et al.~\cite{NewProSet78} defined the \emph{dual} of the structure function $\phi\colon\{0,1\}^n\to\{0,\ldots,m\}$ as the structure function $\phi^D\colon\{0,1\}^n\to\{0,\ldots,m\}$ defined from $\phi$ by
$$
\phi^D(\bfx) ~=~ m-\phi(\mathbf{1}-\bfx){\,},\qquad\bfx\in\{0,1\}^n.
$$

The following result provides a very simple conversion formula between the Boolean decompositions of $\phi$ and $\phi^D$.

\begin{proposition}\label{eq:du43}
We have $(\phi^D)_{\ang{k}}=(\phi_{\ang{m-k+1}})^D$ for any $k\in\{1,\ldots,m\}$.
\end{proposition}

\begin{proof}
We have
\begin{eqnarray*}
(\phi^D)_{\ang{k}}(\bfx)~=~1 &\Leftrightarrow & \phi^D(\bfx)~\geqslant~k\\
&\Leftrightarrow & \phi(\mathbf{1}-\bfx)~\leqslant m-k\\
&\Leftrightarrow & \phi_{\ang{m-k+1}}(\mathbf{1}-\bfx) ~=~ 0\\
&\Leftrightarrow & (\phi_{\ang{m-k+1}})^D(\bfx)~=~1,
\end{eqnarray*}
which proves the result.
\end{proof}

Using Proposition~\ref{eq:du43} we can retrieve the explicit expression of the structural signature of the dual system as established by Da and Hu \cite[Theorem~7.2.4]{DaHu13}. The following proposition shows that this result actually still holds for the probability signature whenever the joint relative quality function is invariant under the operation of set complement. This is the case for instance when the numbers $q(A,B)$ are given by \eqref{eq:qExch5}.

Let us denote the probability signature and the tail probability signature of the dual system by $\mathbf{p}^D$ and $\overline{\mathbf{P}}{\,}{\!}^D$, respectively.

\begin{proposition}\label{prop:Signdua}
Assume that the joint relative quality function is invariant under the operation of set complement, that is
\begin{equation}\label{eq:passcompl}
q(A,B) ~=~ q(C\setminus A,C\setminus B),\qquad A,B\subseteq C.
\end{equation}
Then we have
\begin{equation}
p^D_{k,l} = p_{n-l+1,n-k+1}{\,},\quad k,l=1,\ldots,n\label{eq:pkld}
\end{equation}
and
\begin{equation}
\overline{P}{\,}{\!}^D_{k,l} = \overline{P}_{n-l,n-k}-\overline{P}_{n-l,0}-\overline{P}_{0,n-k}+\overline{P}_{0,0}{\,},\quad k,l=0,\ldots,n.\label{eq:Pkld}
\end{equation}
\end{proposition}

\begin{proof}
Let us prove the second formula. The first formula can then be derived from the conversion formula \eqref{eq:CoFo2}. By Proposition~\ref{prop:kl3} we can assume that $k\geqslant l$. By Theorem~\ref{thm:main32} we have
$$
\overline{P}{\,}{\!}^D_{k,l} ~=~ \sum_{\textstyle{B'\subseteq C\atop |B'|=n-l}}{\,}\sum_{\textstyle{A'\subseteq B'\atop |A'|=n-k}}q(A',B'){\,}(\phi^D)_{\ang{1}}(A')(\phi^D)_{\ang{2}}(B').
$$
Letting $A=C\setminus A'$ and $B=C\setminus B'$, the latter expression becomes
\begin{multline*}
\overline{P}{\,}{\!}^D_{k,l} ~=~ \sum_{\textstyle{A\subseteq C\atop |A|=k}}{\,}\sum_{\textstyle{B\subseteq A\atop |B|=l}}q(C\setminus A,C\setminus B)\\
\times {\,}(\phi^D)_{\ang{1}}(C\setminus A)(\phi^D)_{\ang{2}}(C\setminus B).
\end{multline*}
Using both Proposition~\ref{eq:du43} and our assumption on the joint relative quality function, we obtain
\begin{multline*}
\overline{P}{\,}{\!}^D_{k,l} ~=~ \sum_{\textstyle{A\subseteq C\atop |A|=k}}{\,}\sum_{\textstyle{B\subseteq A\atop |B|=l}}q(A,B)\\
\times {\,}\big(1-\phi_{\ang{2}}(A)-\phi_{\ang{1}}(B)+\phi_{\ang{2}}(A)\phi_{\ang{1}}(B)\big).
\end{multline*}

Now recall from Proposition~\ref{prop:5re86} that
$$
\sum_{\textstyle{A\subseteq C\atop |A|=k}}{\,}\sum_{\textstyle{B\subseteq A\atop |B|=l}}q(A,B) ~=~ 1 ~=~ \overline{P}_{0,0}.
$$
Also, denoting the tail structure signature of $\phi_{\ang{2}}$ by $\overline{\mathbf{P}}_{\ang{2}}$ and using \eqref{eq:f5sff}, we have
\begin{multline*}
\sum_{\textstyle{A\subseteq C\atop |A|=k}}\phi_{\ang{2}}(A){\,}\sum_{\textstyle{B\subseteq A\atop |B|=l}}q(A,B) ~=~ \overline{P}_{n-k,\ang{2}}\\
~=~ \Pr(T_{\calS}^{\geqslant 2}>T_{n-k:n}) ~=~ \overline{P}_{0,n-k}{\,}.
\end{multline*}
Similarly, we obtain
\begin{multline*}
\sum_{\textstyle{B\subseteq C\atop |B|=l}}\phi_{\ang{1}}(B){\,}\sum_{\textstyle{A\supseteq B\atop |A|=k}}q(A,B) ~=~ \overline{P}_{n-l,\ang{1}}\\
~=~ \Pr(T_{\calS}^{\geqslant 1}>T_{n-l:n}) ~=~ \overline{P}_{n-l,0}{\,}.
\end{multline*}
We then use Theorem~\ref{thm:main32} again and collect the terms to complete the proof.
\end{proof}

\begin{remark}
\begin{enumerate}
\item[(a)] We observe that, independently of any assumption, Equations~\eqref{eq:pkld} and \eqref{eq:Pkld} are equivalent (due to the conversion formulas \eqref{eq:CoFo1} and \eqref{eq:CoFo2}). This means that \eqref{eq:pkld} holds if and only if \eqref{eq:Pkld} holds.

\item[(b)] Condition \eqref{eq:passcompl} might be surprising. However, this condition is not only sufficient for Eqs.~\eqref{eq:pkld} and \eqref{eq:Pkld} to hold, but it is also necessary in the sense that if we ask \eqref{eq:pkld} and \eqref{eq:Pkld} to hold for every semicoherent system, then \eqref{eq:passcompl} must hold. To see that this claim holds, just apply \eqref{eq:Pkld} to the system $\phi=\phi_{\ang{1}}+\phi_{\ang{2}}$, where $\phi_{\ang{1}}(\bfx)=\bigvee_{i\in B}x_i$ and $\phi_{\ang{2}}(\bfx)=\bigvee_{i\in A}x_i$ for some $A,B\subseteq C$ such that $A\subseteq B$.
\item[(c)] It is not difficult to exhibit examples to show that, as soon as $n\geqslant 3$, condition \eqref{eq:passcompl} is weaker than requiring that $q=q_0$ (where $q_0$ is defined in \eqref{eq:qExch5}).
\end{enumerate}
\end{remark}

\section{Some concluding remarks about multistate systems}

In Sections 4 and 5 we have investigated the concepts of probability signature and reliability function for multistate systems made up of two-state components. To this extent we have stated a Decomposition Principle based on the Boolean decomposition of the structure functions and this decomposition has revealed the interesting fact that the investigation of a multistate system reduces to the simultaneous investigation of several two-state systems.

Although the use of the Boolean decomposition in the investigation of multistate systems dates back to the work by Block and Savits \cite{BloSav82} in 1982 and possibly earlier, the fact that a multistate system can be decomposed into several two-state systems on the same set of components has, up to our knowledge, never been exploited to investigate the signature of multistate systems. For instance, Gertsbakh et al.~\cite[Theorem 1]{GerSchSpi12} and Da and Hu~\cite[Theorem 7.2.3]{DaHu13} recently obtained a signature-based decomposition of the reliability of a $3$-state system in the i.i.d.\ case. However, their proofs strongly rely on both the regularity condition (see Remark~\ref{rem:reg5} above) and the i.i.d.\ character of the component lifetimes. In the present paper the Decomposition Principle enabled us to use the results of Sections 2 and 3 to derive Proposition~\ref{prop:mainDEC9} almost immediately and without any restrictive condition. This shows that the Decomposition Principle offers a more general algebraic framework and a very efficient tool to deal with such problems in the general non-i.i.d.\ case.

Note that the system components considered in this paper have two states only. Of course it is natural to also investigate multistate systems made up of multistate components (see, e.g., \cite{BarWu78,BloSav82,Ery14,Ery15,Hat79,Jan85,Lev13,LiWuLaiLiu05}). We believe that our algebraic approach can be extended to this general case to investigate the concept of signature. This constitutes a topic of ongoing research work.

\section*{Acknowledgments}

The authors thank the anonymous reviewers for pointing out Natvig's work \cite{Nat82,Nat11}. Jean-Luc Marichal acknowledges partial support by the research project R-AGR-0500 of the University of Luxembourg. Jorge Navarro acknowledges support by Ministerio de Econom\'{\i}a y Competitividad under grant MTM2012-34023-FEDER.


\end{document}